\documentclass{article}

\usepackage{amsmath}
\usepackage{amssymb}
\usepackage{graphicx}
\usepackage{epic,eepic}

\usepackage{color}

\usepackage{amstext,amsthm,amssymb,amsmath}
\usepackage{epsfig,subfigure,epstopdf}
\usepackage{graphicx}
\usepackage{subfigure}
\usepackage{wrapfig}
\usepackage{fullpage}
\usepackage{color}
\usepackage{multirow}
\usepackage{tabulary}
\usepackage{booktabs}
\usepackage{enumerate}

\newtheorem{theorem}{Theorem}
\newtheorem{lemma}{Lemma}
\newtheorem{remark}{Remark}

\usepackage[utf8]{inputenc}
\usepackage[english]{babel}
\usepackage{hyperref}

\usepackage{nomencl}

\theoremstyle{definition}

\begin{document}
\title{Constraint Energy Minimizing Generalized Multiscale Finite Element Method}

\author{
Eric T. Chung \thanks{Department of Mathematics,
The Chinese University of Hong Kong (CUHK), Hong Kong SAR. Email: {\tt tschung@math.cuhk.edu.hk}.
The research of Eric Chung is supported by Hong Kong RGC General Research Fund (Project 14317516).}
\and
Yalchin Efendiev \thanks{Department of Mathematics \& Institute for Scientific Computation (ISC),
Texas A\&M University,
College Station, Texas, USA. Email: {\tt efendiev@math.tamu.edu}.}
\and
Wing Tat Leung \thanks{Department of Mathematics, Texas A\&M University, College Station, TX 77843}
}

\maketitle

\begin{abstract}
In this paper, we propose Constraint Energy Minimizing Generalized
Multiscale Finite Element Method (CEM-GMsFEM).
The main goal of this paper
is to design multiscale basis functions within GMsFEM framework such that
the convergence of
method is independent of the contrast and
linearly decreases with respect to mesh size if oversampling size
is appropriately chosen.
We would like to show a mesh-dependent convergence
with a minimal number of basis functions.
Our construction starts with
an auxiliary multiscale space by solving local spectral problems.
In auxiliary multiscale space,
we select the basis functions that correspond to small (contrast-dependent)
eigenvalues. These basis functions represent the channels
(high-contrast features that connect the boundaries of the coarse block).
Using the auxiliary space, we propose a constraint energy minimization to
construct multiscale spaces. The minimization is performed
in the oversampling domain, which is larger than the target coarse block.
The constraints allow handling
non-decaying components of the local minimizers. If the auxiliary space
is correctly chosen, we show that the convergence rate is independent of
the contrast (because the basis representing the channels
are included in the auxiliary space) and is proportional
to the coarse-mesh size (because the constrains handle non-decaying components
of the local minimizers).
The oversampling size weakly depends on the contrast as our analysis
shows.
The convergence theorem requires that channels are not aligned
with the coarse edges, which hold in many applications, where
the channels are oblique with respect to the coarse-mesh geometry.
The numerical results confirm our theoretical results.
In particular, we show that if the oversampling domain size is not
sufficiently large, the errors are large. To remove the
contrast-dependence of the oversampling size, we propose
a modified construction for basis functions and present
numerical results and the analysis.

\end{abstract}

\section{Introduction}

Many practical applications contain multiple scales and high contrast.
These include flows in fractured media, processes in channelized porous
media and so on. Due to scale disparity and the contrast, some type
of coarse-grid models are used to solve these problems.
The coarse grid is typically much larger than the fine-grid size
and it (the coarse grid)
 contains many heterogeneities and high contrast. In modeling
and simulations of multiscale problems,
it is difficult to adjust coarse-grid sizes based
on scales and contrast. Thus, it is important that the numerical
performance is independent of these physical parameters.

There have been many existing approaches in the literature to handle
multiscale problems. In this paper, we focus on
Darcy flow equation in heterogeneous media.
These multiscale approaches include
homogenization approaches \cite{bour84,dk92,cp95},
numerical upscaling methods \cite{dur91,cdgw03,weh02,durlofsky2003upscaling},
multiscale finite element methods \cite{hw97},
variational multiscale methods
\cite{hughes95,bfhr97,hfmq98, npp08, jp05},
heterogeneous multiscale methods \cite{ee03,abdul_yun,ming2007analysis,abdulle2012heterogeneous,engquist2013heterogeneous},
mortar multiscale methods \cite{pwy02, apwy07, pes05},
localized orthogonal decomposition
methods \cite{maalqvist2014localization},
equation-free approaches \cite{kghkrt03, rk07, pk07,lkgk07},
generalized multiscale finite element methods
\cite{egh12,chung2016adaptive,chung2015generalizedwave}
 and so on.
Some of these approaches are based on homogenization methods
and compute effective properties.
Once the effective properties are computed,
the global problem is solved on the coarse grid.
Our methods are in the class of multiscale finite element methods,
where we seek multiscale basis functions to represent the local
heterogeneities. In multiscale methods, one constructs
multiscale basis functions that can capture the local oscillatory
behavior of the solution.

Our approaches are based on Generalized Multiscale Finite Element Method
(GMsFEM), \cite{egh12,chung2016adaptive,chung2015generalizedwave}.
 This approach, as MsFEM, constructs multiscale basis functions
in each coarse element via local spectral problems. Once local
snapshot space is constructed, the main idea of the GMsFEM is to
solve local spectral problems and identify multiscale basis functions.
These approaches share some common elements with multi-continuum approaches
and try to identify high-contrast features that need to be represented
individually. These non-local
features are typically channels (high-contrast regions
that connect the boundaries of the coarse grid) and need separate (individual)
basis functions. These observations about representing channels
separately are consistent with multi-continuum
methods; however, GMsFEM provides a general framework for deriving
coarse-grid equations.
We note that
the localizations of channels are not possible,
in general, and this is the reason for
constructing basis functions for channels separately as discussed in
\cite{egw10,ge09_1reduceddim}.
These ideas are first used in
designing optimal preconditioners \cite{ge09_1reduceddim}.
In GMsFEM, the local spectral problems and snapshots, if identified
appropriately, correctly identify
the necessary channels without any geometric
interpretation.

It was shown that the GMsFEM's convergence depends on the eigenvalue
decay \cite{chung2014adaptive}.
 However, it is difficult to show a coarse mesh dependent
convergence without using oversampling and many basis functions.
In this paper, we would like to show a mesh-dependent convergence
with a minimal number of basis functions.
The convergence analysis of the GMsFEM suggests that
one needs to include eigenvectors corresponding to small eigenvalues
in the local spectral decomposition. We note that these small
eigenvalues represent the channelized features, as we discussed above.
 To
obtain a mesh-dependent convergence, we use the ideas
from \cite{owhadi2014polyharmonic, maalqvist2014localization, hou2017}\footnote{We learned about \cite{hou2017} in IPAM workshop (April 2017), which is similar to
Section \ref{sec:ms}  (and Section \ref{sec:anal})
and done independently and earlier
 by Tom Hou and Pengchuan Zhang.},
which consists of using oversampling domains and obtaining decaying
local solutions. For high-contrast problems, the local solutions
do not decay in channels and thus, we need approaches that can
take into account the information in the channels when constructing
the decaying local solutions.

The proposed approach starts with auxiliary multiscale basis functions
constructed using the GMsFEM in each coarse block.
This auxiliary space contains the information related to channels
and the number of these basis functions is the same as the number
of the channels, which is a minimal number of basis
functions required representing high-contrast features.
 This auxiliary space is used to take care of the
non-decaying component of the oversampled local solutions, which
occurs in the channels. The construction of multiscale basis functions
is done by seeking
 a minimization of a functional subject to a constraint such that
the minimizer is orthogonal (in a certain sense)
 to the auxiliary space. This
allows handling non-decaying component of the oversampled local solutions.
The resulting approach contains several basis functions per element
and one can use an adaptivity (\cite{chung2014adaptive})
to define the basis functions.
This construction allows obtaining the convergence rate $H/\Lambda$,
where $\Lambda$ is the minimal eigenvalue that the corresponding
eigenvector is not included in the space. Our analysis
also shows that the size of the oversampling domain depends on
the contrast weakly (logarithmically).
To remove the contrast dependence of the oversampling domain size,
we propose a modified algorithm. In this algorithm, we use
the same auxiliary space; however, the minimization is done by relaxing
the constraint.

In the paper, we present numerical results for two heterogeneous
permeability fields.
In both cases, the permeability fields contain channels and inclusions
with high conductivity values. We select  auxiliary basis
functions such that to include all channelized features (i.e.,
the eigenvectors corresponding to very small (contrast-dependent)
eigenvalues).  Our numerical results show that the error decays
as we decrease the coarse-mesh size; however, this is sensitive
to the oversampling domain size. We present numerical results
that show that if there is not sufficient oversampling, the errors
are large and contrast dependent. Furthermore, we also present the numerical
results for our modified algorithm and show that this contrast
 dependence
is removed and the oversampling domain sizes are less sensitive to
the contrast.

The paper is organized as follows. In Section \ref{sec:prelim},
we present some preliminaries. In Section \ref{sec:ms}, we present
the construction of multiscale basis functions.
We present the analysis of the approach in Section \ref{sec:anal}.
In Section \ref{sec:num},
we present numerical results. In Section \ref{sec:mod},
we discuss an extension and show a modified basis construction.
The conclusions are presented in Section \ref{sec:conclusions}.

\section{Preliminaries}
\label{sec:prelim}

We consider
\begin{equation}
-\mbox{div}\big(\kappa(x)\,\nabla u\big)=f\quad\text{in}\quad\Omega \subset \mathbb{R}^d,\label{eq:original}
\end{equation}
where $\kappa$ is a high-contrast with $\kappa_0 \leq \kappa(x) \leq \kappa_1$
and is a multiscale field. The above equation is
subjected to the boundary conditions $u=0$ on
$\partial\Omega$.
Next, the notions of fine and coarse grids are introduced.
Let
$\mathcal{T}^{H}$ be a conforming partition
of  $\Omega$ into finite elements.
Here, $H$ is the coarse-mesh size and this partition is called
coarse grid.
We let $N_c$ be the number of vertices and $N$ be the number of elements
in the coarse mesh.
We assume that
each coarse element is partitioned into
a connected union of fine-grid blocks
and this partition is called $\mathcal{T}^{h}$.
Note that $\mathcal{T}^{h}$
is a refinement of the coarse grid $\mathcal{T}^{H}$
with the mesh size $h$.
It is assumed that the fine grid is sufficiently fine to resolve
the solution.
An illustration of the fine grid, coarse grid, and
oversampling domain are shown
in Figure \ref{fig:illustration}.

\begin{figure}[ht!]
\centering
\includegraphics[width=3in, height=3in]{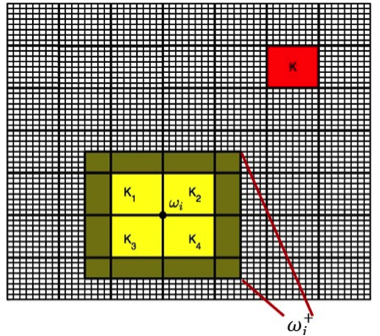}
\caption{Illustration of the coarse grid, fine grid and oversampling domain.}
\label{fig:illustration}
\end{figure}

We
let $V = H^1_0(\Omega)$. Then the solution $u$ of (\ref{eq:original}) satisfies
\begin{equation}
\label{eq:finesol}
a(u,v) = \int_{\Omega} fv \text{ for all } v\in V
\end{equation}
where $a(u,v)=\int_{\Omega}\kappa \nabla u \cdot \nabla v$.
We will discuss the construction of multiscale basis functions in the next section.
We consider $V_{ms}$ to be the space spanned by all multiscale basis functions.
Then the multiscale solution $u_{ms}$ is defined
 as the solution of the following problem, find $u_{ms}\in V_{ms}$ such that
\begin{equation}
\label{eq:mssol}
a(u_{ms},v) = \int_{\Omega} fv \text{ for all } v\in V_{ms}.
\end{equation}
Notice that, in order to show the performance of our method, we will compute
the solution of (\ref{eq:finesol}) on a fine mesh, which is fine enough to resolve the heterogeneities of the true solution $u$.
Moreover, the construction of the multiscale basis functions is also performed on the fine mesh, even
though the definition of the multiscale basis functions is constructed in the space $V$.
We will give the details in the next section.

The computation of the multiscale basis functions is divided
into two stages.
The first stage consists of
constructing the auxiliary multiscale basis functions by using
the concept of generalized multiscale finite element method (GMsFEM).
The next step is the construction of the multiscale basis functions.
In this step, a constrained energy minimizing is performed in the
oversampling
domain.
The construction of the multiscale basis function will be discussed
 in Section \ref{sec:ms}
In the next subsection, we will first introduce the basic concepts of GMsFEM.

\subsection{The basic concepts of Generalized Multiscale Finite Element Method}

Generalized Multiscale Finite Element Method (GMsFEM)
uses
two stages: offline and online. In the offline stage,
a small dimensional finite element space is constructed to solve
the global problem for any input parameter, such as a right-hand side
or boundary condition, on a coarse grid.

The snapshot space,  $V_{H,\text{snap}}^{(i)}$ is constructed
for a generic domain
$\omega_i$ or $K_i$.
For simplicity, we use the notation $\omega_i$, though
multiscale basis functions in Section \ref{sec:ms} will
be constructed in $K_i$.
The snapshot solutions are used to compute multiscale basis functions.
 The appropriate snapshot space
(1) can provide a faster convergence, (2)
can provide
problem relevant
restrictions on the coarse spaces (e.g., divergence free solutions)
and (3) can reduce the cost associated with constructing the offline spaces.

One can use various snapshot spaces (see \cite{chung2016adaptive}), which are
(1) all fine-grid functions; (2) harmonic snapshots; (3) oversampling
harmonic snapshots; and (4) foce-based snapshots. Here, we briefly discuss
harmonic snapshots in oversampling domain.

We briefly discuss the snapshot space that consists of harmonic
extensions of fine-grid functions that are defined on
the boundary of $\omega_i$.
For each fine-grid function, $\delta_l^h(x)$,
which is defined by
$\delta_l^h(x_k)=\delta_{l,k},\,\forall  x_k\in \textsl{J}_{h}(\omega_i)$, where
$\textsl{J}_{h}(\omega_i)$ denotes the set of fine-grid boundary nodes on $\partial\omega_i$, we obtain a snapshot function $\eta_l^{(i)}$ by
\[
\mathcal{L}( \eta_{l}^{(i)})=0\ \ \text{in} \ \omega_i
\]
with the boundary condition, $ \eta_{l}^{(i)}=\delta_l^h(x)$
on $\partial\omega_i$,
and  $\delta_{l,k} = 1$ if $l=k$ and $\delta_{l,k}=0$ if $l\ne k$.
We remark that the snapshot functions can be computed in
the oversampling region
 $\omega_i^{+}$.
In this case,
for each fine-grid function, $\delta_l^h(x)$,
$\delta_l^h(x_k)=\delta_{l,k},\,\forall x_k\in \textsl{J}_{h}(\omega_i^{+})$,
where $\textsl{J}_{h}(\omega_i^{+})$ denotes the set of fine-grid
boundary nodes on $\partial\omega_i^{+}$,
we obtain a snapshot function $\eta_l^{(i),+}$ by
\[
\mathcal{L}( \eta_{l}^{(i),+})=0\ \ \text{in} \ \omega_i^{+}
\]
with $ \eta_{l}^{(i),+}=\delta_l^h(x)$
on $\partial\omega_i^{+}$.
Finally, we remark that one can use randomized boundary conditions
to reduce the computational cost associated with the snapshot construction
\cite{chung2016adaptive, randomized2014}.



The offline space, $V_{ms}^{(i)}$ is computed for each
$\omega_i$
 (with elements of the space denoted $\psi_l^{(i)}$).
We perform a spectral decomposition in the snapshot space and select
the dominant
(corresponding to the smallest eigenvalues)
to construct the offline (multiscale) space.
The convergence rate of the resulting method is proportional
to $1/\Lambda_*$, where
$\Lambda_*$ is the smallest eigenvalue that the corresponding eigenvector
is not included in the multiscale space.
We would like to select local spectral problem such that
we can remove many small eigenvalues with fewer multiscale basis
functions.

The spectral problem depends on the analysis.
In the analysis,
the error is decomposed into coarse subdomains.
The energy functional corresponding to
the domain $\Omega$ is denoted by $a_\Omega(u,u)$, e.g.,
$a_\Omega(u,u) = \int_\Omega \kappa \nabla u\cdot \nabla u$.
Then,
\begin{equation}
\begin{split}
a_\Omega(u-u_H,u-u_H)\preceq
\sum_\omega a_\omega(u^\omega-u_H^\omega,u^\omega-u_H^\omega),
\end{split}
\end{equation}
where $\omega$ are coarse regions ($\omega_i$),
$u^\omega$ is the localization of the solution.
The local spectral problem is chosen to bound
$ a_\omega(u^\omega-u_H^\omega,u^\omega-u_H^\omega)$.
We seek
the subspace $V_{ms}^{\omega}$
such that for any
$\eta\in V_{H,\text{snap}}^{\omega}$,
there exists
$\eta_0\in V_{ms}^{\omega}$ with,
\begin{equation}
\label{eq:off1}
a_{\omega}(\eta-\eta_0,\eta-\eta_0)\preceq {\delta}
s_{\omega}(\eta-\eta_0,\eta-\eta_0),
\end{equation}
where
$s_{\omega}(\cdot,\cdot)$ is an auxiliary bilinear form.
The auxiliary bilinear form needs to be chosen such that the solution
is bounded in the corresponding norm.
Below, we will use a bilinear form defined using the mass matrix.

\section{The construction of the multiscale basis functions}
\label{sec:ms}

In this section, we will present the construction of the auxiliary multiscale basis functions. As we mentioned before, we will use the concept of GMsFEM to construct our auxiliary multiscale basis functions, which will be constructed for each coarse block $K$ in the coarse grid.
Let $K_i$ be the $i$-th coarse block and let $V(K_i)$ be the restriction of $V$ on $K_i$.
Recall that for (\ref{eq:off1}), we need
a local spectral problem, which is to find a real number $\lambda^{(i)}_j$ and a function $\phi_j^{(i)} \in V(K_i)$ such that
\begin{equation}\label{spectralProblem_GMsFEM}
a_{i}(\phi^{(i)}_j, w) = \lambda^{(i)}_j s_{i}(\phi^{(i)}_j, w), \qquad \forall w \in V(K_i),
\end{equation}
where $a_{i}$ is a symmetric non-negative definite bilinear operator
 and $s_{i}$ is a symmetric positive definite bilinear operators defined on $V(K_i) \times V(K_i)$.
 We remark that the above problem is solved on the fine mesh in the actual computations.
Based on our analysis, we can choose
\begin{equation*}
\label{eq:inner}
\begin{aligned}
a_{i}(v, w) = \int_{K_i} \kappa \nabla v \cdot \nabla w, \
s_{i}(v, w) = \int_{K_i} \widetilde{\kappa}  v   w,
\end{aligned}
\end{equation*}
where $\widetilde{\kappa} =  \sum_{j=1}^{N_c}  \kappa|\nabla \chi_j^{ms}|^2 $
and $\{ \chi_j^{ms} \}_{j=1}^{Nc}$ are the standard multiscale finite element (MsFEM) basis functions
(see \cite{hw97}), which satisfy the partition of unity property.
We let $\lambda_j^{(i)}$ be the eigenvalues
of (\ref{spectralProblem_GMsFEM}) arranged in ascending order.
We will use the first $l_i$ eigenfunctions to construct our
local auxiliary multiscale space $V_{aux}^{(i)}$, where
$V_{aux}^{(i)}=\text{span}\{\phi_j^{(i)}|\;j\leq l_i\}$.
The global auxiliary multiscale space $V_{aux}$ is the sum of these local auxiliary multiscale space, namely $V_{aux}=\oplus_{i=1}^N V^{(i)}_{aux}$.
This space is used to construct multiscale basis functions that are
$\phi$-orthogonal to the auxiliary space as defined above.
The notion of $\phi$-orthogonality will be defined next.

For the local auxiliary multiscale space $V_{aux}^{(i)}$, the bilinear form $s_i$ in (\ref{eq:inner})
defines an inner product
with norm $\|v\|_{s(K_i)} = s_i(v,v)^{\frac{1}{2}}$. These local inner products and norms provide a natural
definitions of inner product and norm for the global auxiliary multiscale space $V_{aux}$, which are defined by
\begin{equation*}
s(v,w) = \sum_{i=1}^{N} s_i(v,w), \quad
\|v\|_s = s(v,v)^{\frac{1}{2}}, \quad \forall v\in V_{aux}.
\end{equation*}
We note that $s(v,w)$ and $\|v\|_s$ are also an inner product and norm for the space $V$.
Using the above inner product, we can define the notion of $\phi$-orthogonality in the space $V$.
Given a function $\phi_j^{(i)} \in V_{aux}$, we say that a function $\psi \in V$ is $\phi_j^{(i)}$-orthogonal if
\begin{equation*}
s(\psi, \phi_j^{(i)}) = 1, \quad s(\psi, \phi_{j'}^{(i')}) = 0, \; \text{if }  j' \ne j \text{ or } i' \ne i.
\end{equation*}
Now, we let $\pi_i:V\rightarrow V_{aux}^{(i)}$ be the projection with respect to the inner product $s_i(v,w)$. So, the operator $\pi_i$ is given by
\[
\pi_i(u)= \sum_{j=1}^{l_i} \cfrac{s_{i}(u,\phi_{j}^{(i)})}{s_{i}(\phi_{j}^{(i)},\phi_{j}^{(i)})}\phi_{j}^{(i)},   \quad\;\forall u\in V.
\]
In addition, we let $\pi:V\rightarrow V_{aux}$ be the projection with respect to the inner product $s(v,w)$. So, the operator $\pi$ is given by
\[
\pi(u)=\sum_{i=1}^{N} \sum_{j=1}^{l_i} \cfrac{s_{i}(u,\phi_{j}^{(i)})}{s_{i}(\phi_{j}^{(i)},\phi_{j}^{(i)})}\phi_{j}^{(i)},  \quad \;\forall u\in V.
\]
Note that $\pi = \sum_{i=1}^N \pi_i$.

We next present the construction of our multiscale basis functions. For each coarse element $K_i$, we define an oversampled domain $K_{i,m} \subset \Omega$ by enlarging $K_i$ by $m$ coarse grid layers, where $m \geq 1$ is an integer. We next define the multiscale
 basis function $\psi_{j,ms}^{(i)}\in V_{0}(K_{i,m})$ by
 \begin{equation}
 \label{eq:ms-min}
 \psi_{j,ms}^{(i)}=\text{argmin}\Big\{ a(\psi,\psi) \; | \; \psi \in V_{0}(K_{i,m}), \quad \psi \text{ is } \phi_j^{(i)}\text{-orthogonal} \Big\}
 \end{equation}
 where $V(K_{i,m})$ is the restriction of $V$ in $K_{i,m}$, and $V_{0}(K_{i,m})$ is the subspace of $V(K_{i,m})$ with zero trace on $\partial K_{i,m}$.
Our multiscale finite element space $V_{ms}$ is defined by
\[
V_{ms}=\text{span}\Big\{\psi_{j,ms}^{(i)} \; | \; 1 \leq j \leq l_i, \; 1 \leq i \leq N \Big\}.
\]
The existence of the solution of the above minimization problem will be proved in Lemma \ref{lem:infsup},
where a more general version is considered.
We remark that the linear independence of the above basis functions $ \psi_{j,ms}^{(i)}$ is obvious.


The following are the main ideas behind this multiscale basis function construction.
\begin{itemize}

\item The   $\phi$-orthogonality of the multiscale basis functions allows
a spatial decay, which is one of the contributing factors
of a mesh-size dependent convergence.

\item The multiscale basis functions minimize the energy, which
is important and, in particularly, for the decay.

\item We note if we do not choose an appropriate auxiliary space, this
will affect the convergence rate and the decay rate in terms of the contrast.

\end{itemize}

In Figure~\ref{fig:basis}, we illustrate the importance of the auxiliary space
on the decay of multiscale basis function.
We consider a high-contrast channelized medium as shown in the left plot in Figure~\ref{fig:basis}.
In the middle plot of Figure~\ref{fig:basis}, we show a multiscale basis
with the use of only one eigenfunction in the auxiliary space. We see that the basis function has almost no decay.
On the other hand, in the right plot of Figure~\ref{fig:basis}, we show a multiscale basis
with the use of $4$ eigenfunctions, and we see clearly that the basis function has very fast decay outside the coarse block.

\begin{figure}[ht!]
\centering
\includegraphics[scale=0.3]{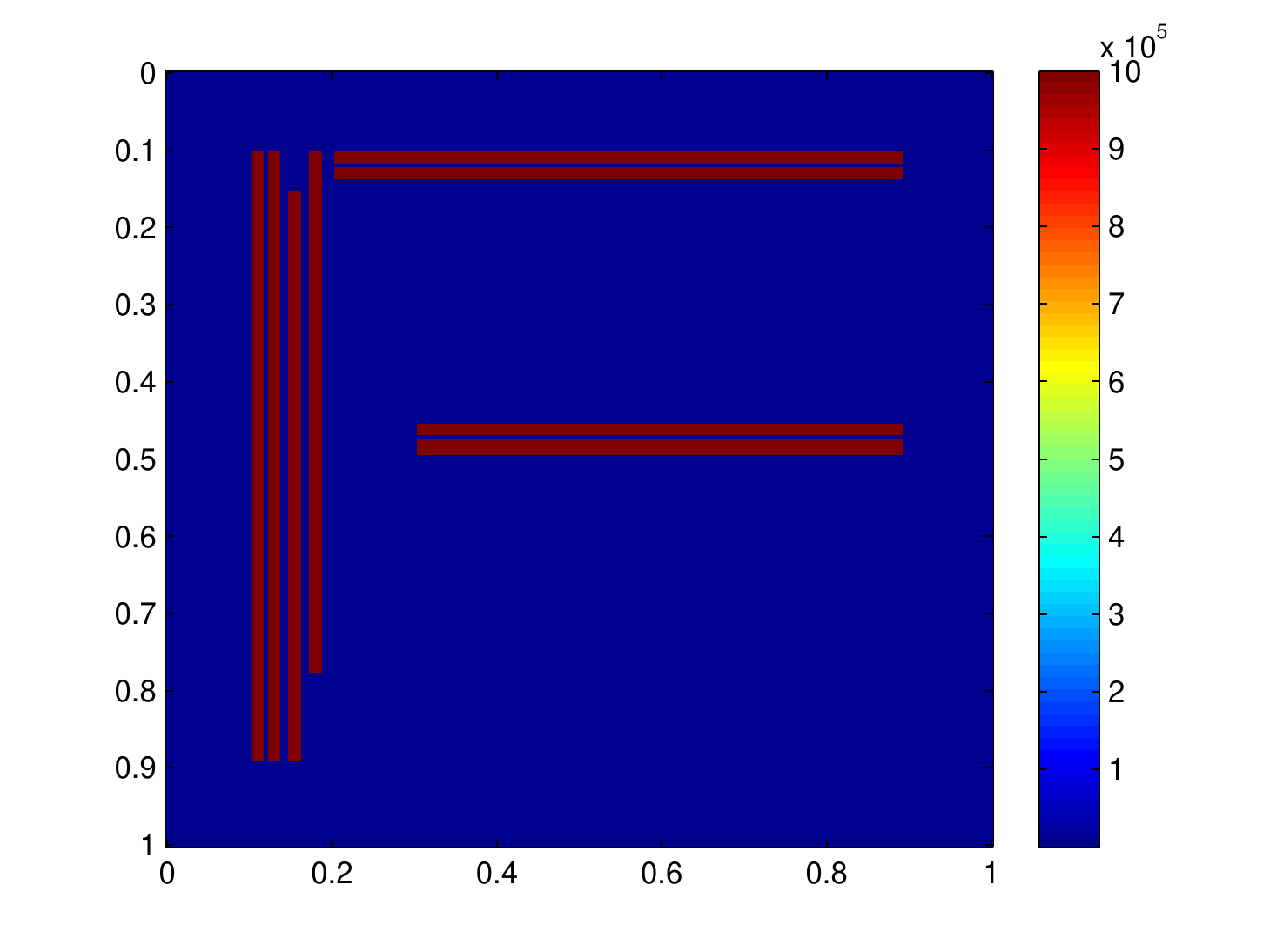}
\includegraphics[scale=0.3]{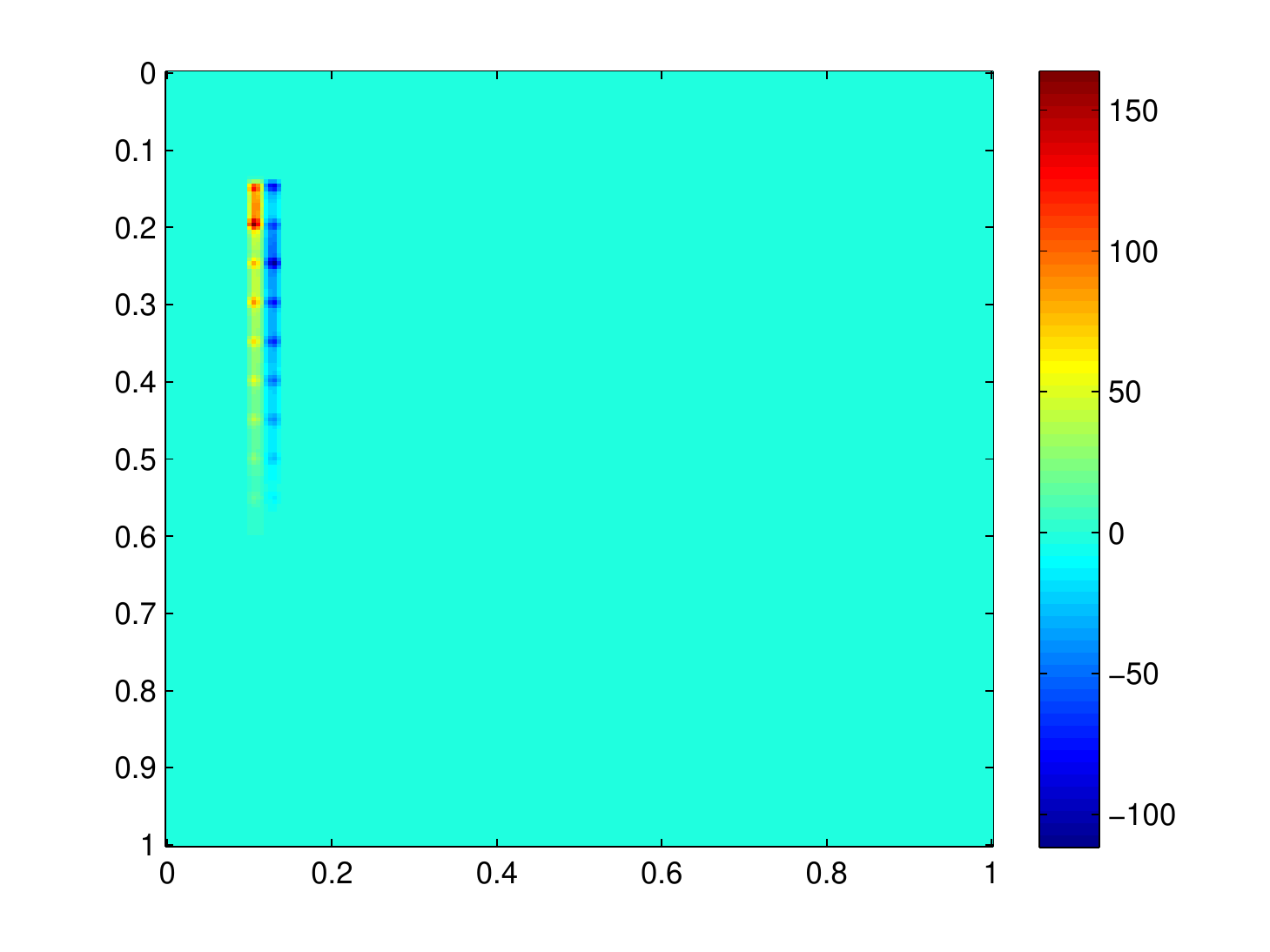}
\includegraphics[scale=0.3]{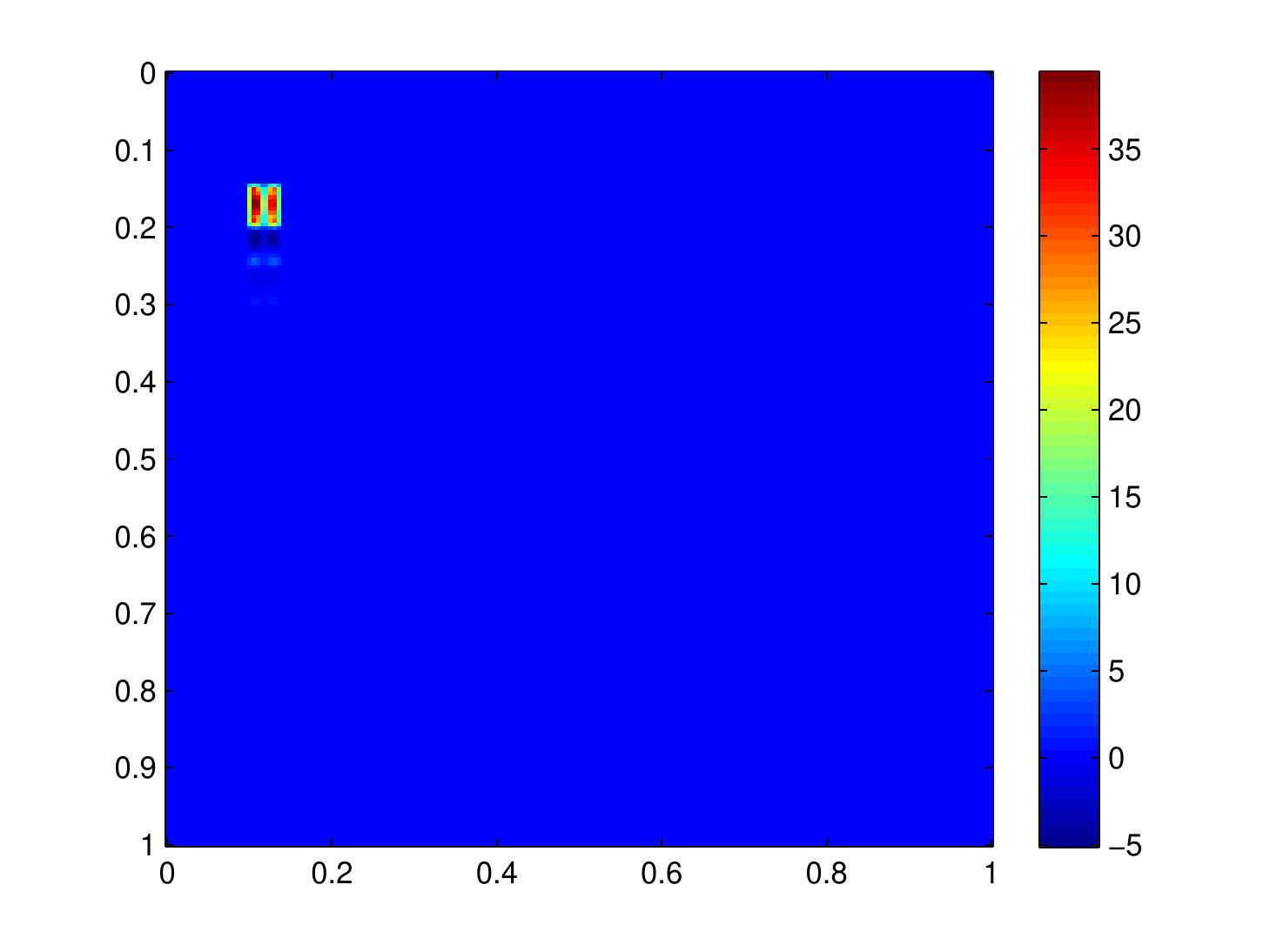}
\caption{An illustration of the decay property of multiscale basis functions. Left: a high contrast medium. Middle: a multiscale basis function
using one eigenfunction in each local auxiliary space. Right: a multiscale basis function using $4$ eigenfunctions in each local auxiliary space.}
\label{fig:basis}
\end{figure}
\begin{remark}
The local multiscale basis construction is motivated by the global
basis construction as defined below. The global basis functions
are used in the convergence analysis.
We will present the construction of the global basis functions.
The global multiscale
 basis function $\psi_{j}^{(i)}\in V$ is defined by
 \begin{equation}
 \label{eq:glo-min}
 \psi_{j}^{(i)}=\text{argmin}\Big\{ a(\psi,\psi) \; | \; \psi \in V, \quad \psi \text{ is } \phi_j^{(i)}\text{-orthogonal} \Big\}.
 \end{equation}
Our multiscale finite element space $V_{glo}$ is defined by
\[
V_{glo}=\text{span}\Big\{\psi_{j}^{(i)} \; | \; 1 \leq j \leq l_i, \; 1 \leq i \leq N \Big\}.
\]
This global multiscale finite element space $V_{glo}$ satisfies a very important orthogonality property,
which will be used in our convergence analysis.
In particular, we define $\tilde{V}_{h}$ as the null space of the projection $\pi$, namely,
$\tilde{V}=\{v\in V \; | \; \pi(v)=0\}$. Then for any $\psi_{j}^{(i)} \in V_{glo}$, we have
\begin{equation*}
a(\psi_j^{(i)},v) = 0, \quad \forall v\in \tilde{V}.
\end{equation*}
Thus, $\tilde{V}\subset V_{glo}^{\bot}$, where $V_{glo}^{\bot}$ is the orthogonal complement of $V_{glo}$ with respect to
the inner product defined using the bilinear form $a$.
Since $\text{dim}(V_{glo})=\text{dim}(V_{aux})$,
we have $\tilde{V}=V_{glo}^{\bot}$. Thus, we have $V=V_{glo}\oplus\tilde{V}$.
In Figure~\ref{fig:glo_basis}, we illustrate the decay of the basis function.


\begin{figure}[ht!]
\centering
\includegraphics[scale=0.3]{medium2}
\includegraphics[scale=0.3]{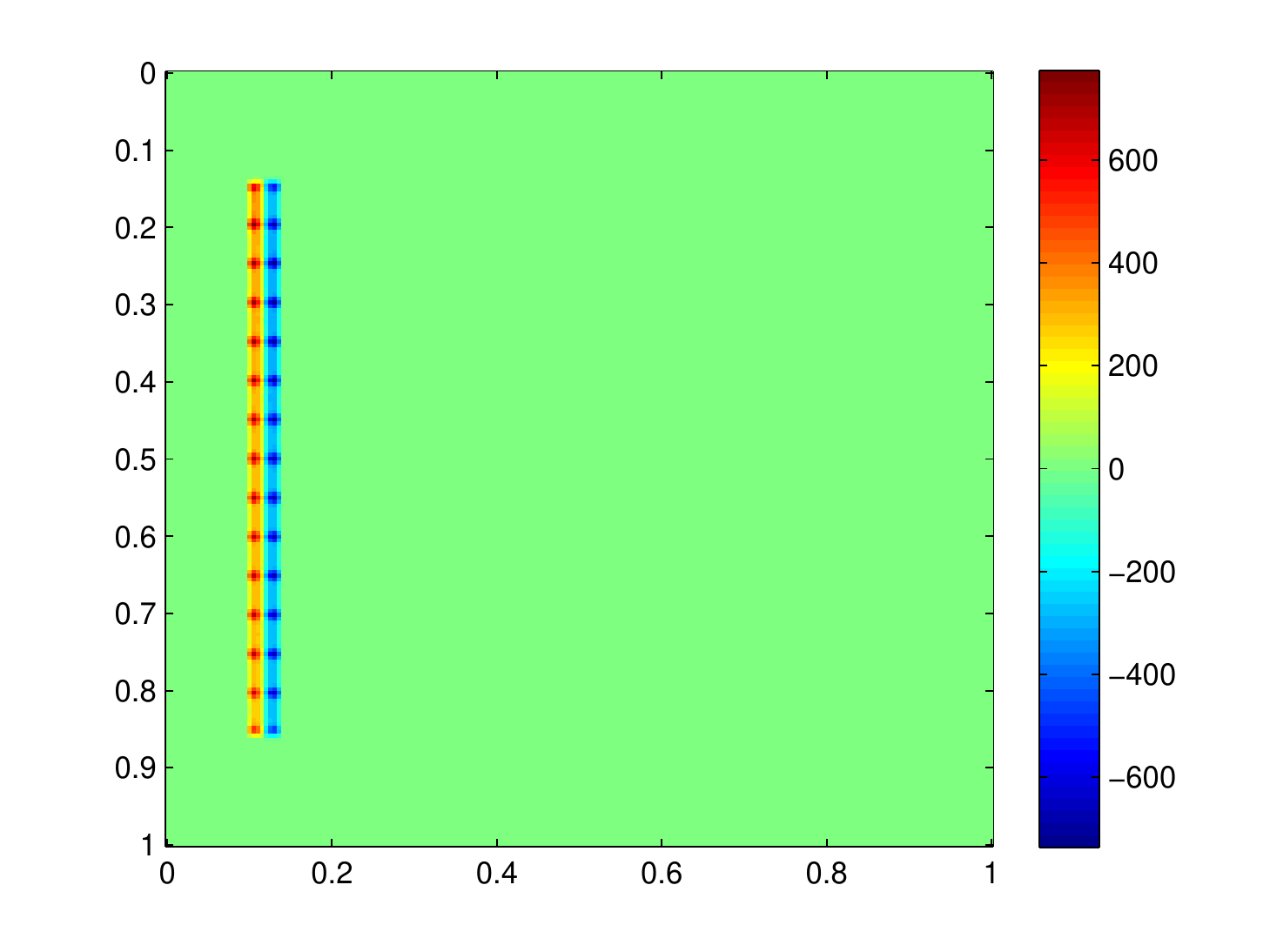}
\includegraphics[scale=0.3]{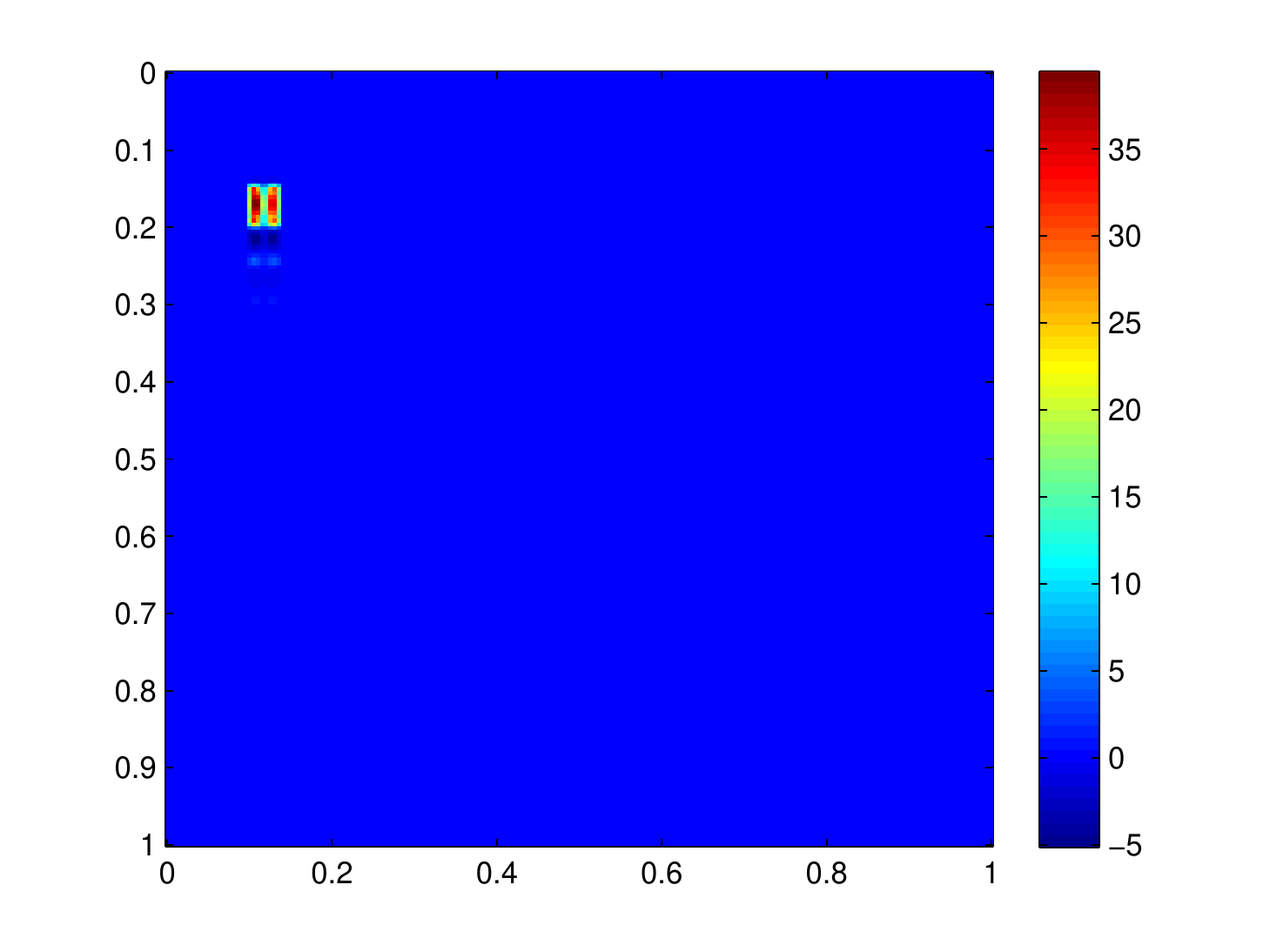}
\caption{An illustration of the decay property of global multiscale basis functions. Left: a high contrast medium. Middle: a multiscale basis function
using one eigenfunction in each local auxiliary space. Right: a multiscale basis function using $4$ eigenfunctions in each local auxiliary space.}
\label{fig:glo_basis}
\end{figure}
\end{remark}

\section{Analysis}
\label{sec:anal}

In this section, we will prove the convergence of our proposed method.
Before proving the convergence of the method, we need to define some notations. We will define two different norms for the finite element space $V$.
One is the $a$-norm $\|\cdot\|_a$ where $\|u\|_a^2=\int_\Omega \kappa | \nabla u|^2$. The other is $s$-norm $\|\cdot\|_s$ where $\|u\|_s^2=\int_\Omega \tilde{\kappa} u^2$.
For a given subdomain $\Omega_i \subset \Omega$, we will define the local $a$-norm and $s$-norm by $\|u\|_{a(\Omega_i)}^2=\int_{\Omega_i} \kappa | \nabla u|^2$ and $\|u\|_{s(\Omega_i)}^2=\int_{\Omega_i} \tilde{\kappa}  u^2$.

To prove the convergence result of the proposed method, we will first show the convergence result of using the global multiscale basis functions. Next, we will give an estimate of the difference between the global basis functions and the multiscale basis functions to show the convergent rate of the proposed method is similar to using global basis functions.
The approximate solution $u_{glo} \in V_{glo}$ obtained in the global multiscale space $V_{glo}$ is defined
by
\begin{equation}
\label{eq:glosol}
a(u_{glo},v) = \int_{\Omega} fv \text{ for all } v\in V_{glo}.
\end{equation}
The convergence analysis will start with the following lemma.

\begin{lemma}
\label{lem:glosol}
Let $u$ be the solution in (\ref{eq:finesol}) and $u_{glo}$ be
the solution of (\ref{eq:glosol}). We have $u-u_{glo}\in\tilde{V}$
and
\[
\|u-u_{glo}\|_{a}\leq \Lambda^{-\frac{1}{2}}\| \tilde{\kappa}^{-\frac{1}{2}} f\|_{L^{2}(\Omega)}
\]
where
\begin{equation*}
\Lambda = \min_{1\leq i \leq N} \lambda^{(i)}_{l_i+1}.
\end{equation*}
Moreover, if we replace the multiscale partition of unity $\{\chi_{j}^{ms}\}$ by
the bilinear partition of unity, we have
\[
\|u-u_{glo}\|_{a}\leq C  H\Lambda^{-\frac{1}{2}} \| \kappa^{-\frac{1}{2}} f\|_{L^{2}(\Omega)}.
\]
\label{lem:global_estimate}
\end{lemma}

\begin{proof}
By the definitions of $u$ and $u_{glo}$, we have
\begin{align*}
a(u,v) & =(f,v), \;\forall v\in V, \label{fine} \\
a(u_{glo},v) & =(f,v), \;\forall v\in V_{glo}.
\end{align*}
Combining these two equations, we obtain
\[
a(u-u_{glo},v)=0,\;\forall v\in V_{glo}.
\]
So, we have $u-u_{glo}\in V_{glo}^{\bot}=\tilde{V}$.
Using this orthogonality property and (\ref{eq:finesol}),
we have
\begin{align*}
a(u-u_{glo},u-u_{glo}) & =a(u,u-u_{glo})
  =(f,u-u_{glo})\\
 & \leq \| \tilde{\kappa}^{-\frac{1}{2}} f \|_{L^2(\Omega)} \, \| u-u_{glo}\|_s.
\end{align*}
Since $u-u_{glo}\in\tilde{V}_h$, we have $\pi(u-u_{glo})=0$.
By the fact that the coarse blocks $K_i$ are disjoint, we also have $\pi_i(u-u_{glo})=0$ for all $i=1,2,\cdots, N$.
Therefore, we have
\begin{align*}
\|u - u_{glo}\|_s^2 = \sum_{i=1}^N\|(u-u_{glo})\|_{s(K_{i})}^{2} & = \sum_{i=1}^N\|(I-\pi_{i})(u-u_{glo})\|_{s(K_{i})}^{2}.
\end{align*}
By using the orthogonality of the eigenfunctions $\phi_j^{(i)}$ of (\ref{spectralProblem_GMsFEM}), we have
\begin{align*}
\sum_{i=1}^N\|(I-\pi_{i})(u-u_{glo})\|_{s(K_{i})}^{2} & \leq
\cfrac{1}{\Lambda}\sum_{i=1}^N\|u-u_{glo}\|_{a(K_{i})}^{2} = \cfrac{1}{\Lambda} \| u - u_{glo}\|_a^2.
\end{align*}
The proof for the second part follows from the fact that $|\nabla \chi_j| = O(H^{-1})$ when $\{\chi_j\}$
is the set of bilinear partition of unity functions.
\end{proof}

We remark that, by using the fact that
\begin{align*}
a(u-u_{glo},u-u_{glo})
  =(f,u-u_{glo}) = (f-\pi f, u-u_{glo}),
\end{align*}
we can, under sufficient regularity assumption on $f$,
improve the above result to
\[
\|u-u_{glo}\|_{a}\leq C  H\Lambda^{-\frac{1}{2}} \| \kappa^{-\frac{1}{2}} (f-\pi f) \|_{L^{2}(\Omega)}.
\]

After proving the above lemma, we have the convergence of the method for using global basis functions. Next, we are going to prove these global basis functions are localizable.
The estimate of the difference between the global basis functions and the multiscale basis functions will require the following lemma.
For each coarse block $K$, we define $B$ to be a bubble function with $B(x)>0$ for all $x\in\text{int}(K)$
and $B(x)=0$ for all $x\in\partial K$.
We will take $B = \Pi_j \chi_j^{ms}$ where the product is taken over all vertices $j$ on the boundary of $K$.
Using the bubble function, we define the constant
\begin{equation*}
C_{\pi}=\sup_{K\in\mathcal{T}^{H},\mu\in V_{aux}}\cfrac{\int_{K}\tilde{\kappa}  \mu^{2}}{\int_{K}B\tilde{\kappa} \mu^{2}}.
\end{equation*}
We also define
\begin{equation*}
\lambda_{max} = \max_{1\leq i \leq N} \max_{1\leq j\leq l_i} \lambda_j^{(i)}.
\end{equation*}
The following lemma considers the following minimization problem defined on a coarse block $K_i$:
 \begin{equation}
 \label{eq:min}
v=\text{argmin}\Big\{ a(\psi,\psi) \; | \; \psi \in V_0(K_i), \quad s_i(\psi,v_{aux}) = 1, \quad s_i(\psi,w) = 0, \;\forall w\in v_{aux}^{\perp} \Big\}
 \end{equation}
 for a given $v_{aux} \in V_{aux}^{(i)}$ with $\|v_{aux}\|_{s(K_i)}=1$,
 where $v_{aux}^{\perp} \subset V_{aux}^{(i)}$ is the orthogonal complement of $\text{span}\{ v_{aux}\}$
 with respect to the inner product $s_i$.
We note that the minimization problem (\ref{eq:min}) is a more general version of (\ref{eq:ms-min}).

\begin{lemma}
\label{lem:infsup}
For all $v_{aux}\in V_{aux}$
there exists a function $v\in V$ such that
\[
\pi(v)=v_{aux},\qquad\|v\|_{a}^{2}\leq D \, \|v_{aux}\|_{s}^{2},\qquad\text{supp}(v)\subset\text{supp}(v_{aux}).
\]
We write $D = C_{\mathcal{T}}C_{\pi}(1+\lambda_{max})$, where
$C_{\mathcal{T}}$ is the square of the maximum number of vertices over all coarse elements.
\label{lem:projection}
\end{lemma}
\begin{proof}
The proof consists of two steps. In the first step,
we will show that the problem (\ref{eq:min}) has a unique solution.
In the second step, we will prove the desired result of the lemma.

\noindent
{\bf Step 1}:

Let $v_{aux}\in V_{aux}^{(i)}$. 
The minimization problem (\ref{eq:min}) is equivalent
to the following variational problem: find $v\in V_{0}(K_i)$ and $\mu \in V_{aux}^{(i)}$ such that
\begin{align}
a_i(v,w) + s_i(w,\mu) &=0, \quad \forall w\in V_{0}(K_i), \label{eq:loc_problem} \\
s_i(v,\nu) &=s_i(v_{aux},\nu), \quad \forall \nu \in V_{aux}^{(i)}. \label{eq:loc_problem2}
\end{align}
Note that, the well-posedness of the minimization problem (\ref{eq:min}) is equivalent to the existence of
a function $v\in V_{0}(K_i)$ such that
\begin{equation*}
s_i(v,v_{aux}) \geq C \|v_{aux}\|_{s(K_i)}^2, \quad
\|v\|_{a(K_i)} \leq C \|v_{aux}\|_{s(K_i)}
\end{equation*}
where $C$ is a constant to be determined.

Note that $v_{aux}$ is supported in $K_i$. We let $v = Bv_{aux}$.
By the definition of $s_i$, we have
\begin{equation*}
s_i(v, v_{aux}) = \int_{K_i} \tilde{\kappa} B v_{aux}^2 \geq C^{-1}_{\pi} \| v_{aux}\|_{s(K_i)}^2.
\end{equation*}
Since $\nabla (Bv_{aux}) = v_{aux}\nabla B + B \nabla v_{aux}$, $|B| \leq 1$ and $|\nabla B|^{2}\leq C_{\mathcal{T}}\sum_{j}|\nabla\chi_{j}^{ms}|^{2}$,
we have
\begin{equation*}
\| v\|_{a(K_i)}^2=\| Bv_{aux}\|_{a(K_i)}^2 \leq C_{\mathcal{T}}  C_{\pi} \|v\|_{a(K_i)} \Big( \| v_{aux}\|_{a(K_i)} + \|v_{aux}\|_{s(K_i)} \Big).
\end{equation*}
Finally, using the spectral problem (\ref{spectralProblem_GMsFEM}), we have
\begin{equation*}
\| v_{aux}\|_{a(K_i)} \leq (\max_{1\leq j \leq l_i} \lambda_j^{(i)}) \| v_{aux}\|_{s(K_i)}.
\end{equation*}
This completes the first step.

\noindent
{\bf Step 2}:

From the above proof, we see that the minimization problem (\ref{eq:min}) has a unique solution $v\in V_{0}(K_i)$.
So, we see that $v$ and $v_{aux}$ satisfy (\ref{eq:loc_problem})-(\ref{eq:loc_problem2}).
From (\ref{eq:loc_problem2}), we see that $\pi_i(v)=v_{aux}$.
We note that the other two conditions in the lemma
follow from the above proof.

\end{proof}

Notice that, we can assume $D\geq 1$ in Lemma \ref{lem:infsup}.

Before estimate the difference between the global and multiscale basis function, we need some notations for the oversampling domain and the cutoff function with respect to these oversampling domains.
For each $K_i$, we recall that $K_{i,m} \subset \Omega$ is the oversampling coarse region by enlarging
$K_{i}$ by $m$ coarse grid layers. For $M>m$, we define $\chi_{i}^{M,m}\in\text{span}\{\chi^{ms}_{j}\}$
such that $0 \leq \chi_i^{M,m} \leq 1$ and
\begin{align}
\chi_{i}^{M,m} & =1\text{ in }K_{i,m}, \label{cutoff1} \\
\chi_{i}^{M,m} & =0\text{ in }\Omega\backslash K_{i,M}. \label{cutoff2}
\end{align}
Note that, we have $K_{i,m} \subset K_{i,M}$. Moreover,
$\chi_i^{M,m}=1$ on the inner region $K_{i,m}$
and $\chi_i^{M,m}=0$ outside the outer region $K_{i,M}$.

The following lemma shows that our multiscale basis functions have
a decay property.
In particular, the multiscale basis functions are small
outside an oversampled region specified in the lemma.

\begin{lemma}
\label{lem:decay}
We consider the oversampled domain $K_{i,k}$ with $k\geq2$.
That is, $K_{i,k}$ is an oversampled region by enlarging $K_i$ by $k$ coarse grid layers.
Let $\phi_j^{(i)} \in V_{aux}$ be a given auxiliary multiscale basis function.
We let $\psi_{j,ms}^{(i)}$ be the multiscale basis functions obtained in (\ref{eq:ms-min})
and let $\psi_{j}^{(i)}$ be the global multiscale basis functions obtained in (\ref{eq:glo-min}).
Then we have
\[
\|\psi_{j}^{(i)}-\psi_{j,ms}^{(i)}\|^2_{a}\leq E\, \|\phi_{j}^{(i)}\|^2_{s(K_i)}
\]
where $E = 8D^2(1+\Lambda^{-1}) \Big(1+\cfrac{\Lambda^{\frac{1}{2}}}{2D^{\frac{1}{2}}} \Big)^{1-k}$.
\end{lemma}
\begin{proof}
For the given $\phi_j^{(i)} \in V_{aux}$,
using Lemma \ref{lem:projection}, there exist a $\tilde{\phi}_{j}^{(i)}\in V_{h}$
such that
\begin{equation}
\label{eq:est}
\pi(\tilde{\phi}_{j}^{(i)})=\phi_{j}^{(i)}, \quad \|\tilde{\phi}_{j}^{(i)}\|_{a}^{2}\leq D\|\phi_{j}^{(i)}\|_{s}^{2} \quad
\text{and} \quad \text{supp}(\tilde{\phi}_{j}^{(i)})\subset K_{i}.
\end{equation}
We let
$\eta=\psi_{j}^{(i)}-\tilde{\phi}_{j}^{(i)}$. Note that $\eta \in\tilde{V}_h$ since $\pi(\eta)=0$.
By using the resulting variational forms of the minimization problems (\ref{eq:ms-min}) and (\ref{eq:glo-min}),
we see that $\psi_{j}^{(i)}$ and $\psi_{j,ms}^{(i)}$ satisfy
\[
a(\psi_{j}^{(i)},v) + s(v,\mu_j^{(i)}) = 0, \quad \; \forall v\in V
\]
and
\[
a(\psi_{j,ms}^{(i)},v) + s(v,\mu_{j,ms}^{(i)})= 0, \quad\; \forall v\in V_{0}(K_{i,k})
\]
for some $\mu_j^{(i)}, \mu_{j,ms}^{(i)} \in V_{aux}$.
Subtracting the above two equations and restricting $v \in \tilde{V}_{0}(K_{i,k})$, we have
\begin{equation*}
a(\psi_{j}^{(i)}-\psi_{j,ms}^{(i)},v) =0 ,\quad \forall v\in \tilde{V}_{0}(K_{i,k}).
\end{equation*}
Therefore, for $v\in \tilde{V}_{0}(K_{i,k})$, we have
\begin{equation*}
\begin{split}
\|\psi_{j}^{(i)}-\psi_{j,ms}^{(i)}\|_{a}^2
&= a(\psi_{j}^{(i)}-\psi_{j,ms}^{(i)},\psi_{j}^{(i)}-\psi_{j,ms}^{(i)}) \\
&= a(\psi_{j}^{(i)}-\psi_{j,ms}^{(i)},\psi_{j}^{(i)}-\tilde{\phi}_j^{(i)}-\psi_{j,ms}^{(i)}+\tilde{\phi}_j^{(i)})
= a(\psi_{j}^{(i)}-\psi_{j,ms}^{(i)}, \eta-v)
\end{split}
\end{equation*}
since $-\psi_{j,ms}^{(i)}+\tilde{\phi}_j^{(i)} \in \tilde{V}_{0}(K_{i,k})$. So, we obtain
\begin{equation}
\label{eq:estimate1}
\|\psi_{j}^{(i)}-\psi_{j,ms}^{(i)}\|_{a}\leq\|\eta-v\|_{a}
\end{equation}
 for all $v\in\tilde{V}_{0}(K_{i,k}).$

Next, we will estimate $\|\psi_{j}^{(i)}-\psi_{j,ms}^{(i)}\|_{a}$.
We consider the $i$-th coarse block $K_i$.
For this block, we consider two oversampled regions $K_{i,k-1}$ and $K_{i,k}$.
Using these two oversampled regions, we define the cutoff function $\chi_{i}^{k,k-1}$
with the properties in (\ref{cutoff1})-(\ref{cutoff2}), where we take $m=k-1$ and $M=k$.
For any coarse block $K_j \subset K_{i,k-1}$, by (\ref{cutoff1}), we have $\chi_i^{k,k-1} \equiv 1$
on $K_j$. Since $\eta \in \tilde{V}$, we have
\begin{equation*}
s_{j}(\chi_{i}^{k,k-1}\eta,\phi_{n}^{(j)})
 = s_{j}(\eta,\phi_{n}^{(j)}) = 0, \quad \forall n=1,2,\cdots, l_j.
\end{equation*}
From the above result and the fact that $\chi_i^{k,k-1} \equiv 0$ in $\Omega\backslash K_{i,k}$, we have
\begin{equation*}
\text{supp}\Big(\pi(\chi_{i}^{k,k-1)}\eta)\Big)\subset K_{i,k}\backslash K_{i,(k-1)}.
\end{equation*}
Using Lemma \ref{lem:projection}, for the function $\pi (\chi_{i}^{k,k-1}\eta)$,
there is $\mu\in V$ such that $\text{supp}(\mu) \subset K_{i,k}\backslash K_{i,k-1}$ and
$\pi(\mu-\chi_{i}^{k,k-1}\eta)=0$. Moreover, also from Lemma \ref{lem:projection},
\begin{align}
\label{eq:estimate3}
\|\mu\|_{a(K_{i,k}\backslash K_{i,k-1})} & \leq D^{\frac{1}{2}}\|\pi(\chi_{i}^{k,k-1}\eta)\|_{s(K_{i,k}\backslash K_{i,k-1})}
  \leq D^{\frac{1}{2}}\|\chi_{i}^{k,k-1}\eta\|_{s(K_{i,k}\backslash K_{i,k-1})}
\end{align}
where the last inequality follows from the fact that $\pi$ is a projection.
Hence, taking $v = \mu-\chi_{i}^{k,k-1}\eta$ in (\ref{eq:estimate1}), we have
\begin{align}
\label{eq:estimate2}
\|\psi_{j}^{(i)}-\psi_{j,ms}^{(i)}\|_{a} \leq \|\eta-v\|_{a} & \leq\|(1-\chi_{i}^{k,k-1})\eta\|_{a}+\|\mu\|_{a(\omega_{i,k}\backslash\omega_{i,k-1})}.
\end{align}
Next, we will estimate the two terms on the right hand side. We will divide the proof
in four steps.

\noindent
{\bf Step 1}: We will estimate the term $\|(1-\chi_{i}^{k,k-1})\eta\|_{a}$ in (\ref{eq:estimate2}). By a direct computation, we have
\begin{align*}
\|(1-\chi_{i}^{k,k-1})\eta\|_{a}^{2} & \leq 2(\int_{\Omega\backslash K_{i,k-1}}\kappa(1-\chi_{i}^{k,k-1})^{2}|\nabla\eta|^{2}
+\int_{\Omega\backslash K_{i,k-1}}\kappa|\nabla\chi_{i}^{k,k-1}|^{2}\eta^{2}).
\end{align*}
Note that, we have $1-\chi_{i}^{k,k-1} \leq 1$.
For the second term on the right hand side of the above inequality,
we will use the fact that $\eta \in \tilde{V}_h$ and the spectral problem (\ref{spectralProblem_GMsFEM}).
Thus, we conclude that
\begin{align*}
\|(1-\chi_{i}^{k,k-1})\eta\|_{a}^{2} & \leq
 2(1+\cfrac{1}{\Lambda})\int_{\Omega\backslash K_{i,k-1}}\kappa|\nabla\eta|^{2}.
\end{align*}
We will estimate the right hand side in Step 3.

\noindent
{\bf Step 2}: We will estimate the term $\|\mu\|_{a(K_{i,k}\backslash K_{i,k-1})}$ in (\ref{eq:estimate2}).
By (\ref{eq:estimate3}) and $|\chi_i^{k,k-1}| \leq 1$, we have
\begin{align*}
\|\mu\|_{a(K_{i,k}\backslash K_{i,k-1})}^{2} & \leq D \|(\chi_{i}^{k,k-1}\eta)\|_{s(K_{i,k}\backslash K_{i,k-1})}^{2}
  \leq \cfrac{D}{\Lambda}\int_{K_{i,k}\backslash K_{i,k-1}}\kappa|\nabla\eta|^{2}.
\end{align*}

Combining Step 1 and Step 2, we obtain
\begin{equation}
\label{eq:estimate4}
\|\psi_{j}^{(i)}-\psi_{j,ms}^{(i)}\|_{a}^2
\leq 2D (1+\cfrac{1}{\Lambda}) \|\eta\|^2_{a(\Omega\backslash K_{i,k-1})}.
\end{equation}

\noindent
{\bf Step 3}: Finally, we will estimate the term $\|\eta\|_{a(\Omega\backslash K_{i,k-1})}$.
We will first show that the following recursive inequality holds
\begin{equation}
\label{eq:recursive}
\|\eta\|_{a(\Omega\backslash K_{i,k-1})}^{2}\leq\Big(1+\cfrac{\Lambda^{\frac{1}{2}}}{2D^{\frac{1}{2}}} \Big)^{-1}\|\eta\|_{a(\Omega\backslash K_{i,k-2})}^{2}.
\end{equation}
where $k-2 \geq 0$. Using (\ref{eq:recursive}) in (\ref{eq:estimate4}), we have
\begin{equation}
\label{eq:estimate5}
\|\psi_{j}^{(i)}-\psi_{j,ms}^{(i)}\|_{a}^2
\leq 2D (1+\cfrac{1}{\Lambda}) \, \Big(1+\cfrac{\Lambda^{\frac{1}{2}}}{2D^{\frac{1}{2}}} \Big)^{-1}\|\eta\|_{a(\Omega\backslash K_{i,k-2})}^{2}.
\end{equation}
By using (\ref{eq:recursive}) again in (\ref{eq:estimate5}), we conclude that
\begin{equation}
\|\psi_{j}^{(i)}-\psi_{j,ms}^{(i)}\|_{a}^2
\leq 2D (1+\cfrac{1}{\Lambda}) \, \Big(1+\cfrac{\Lambda^{\frac{1}{2}}}{2D^{\frac{1}{2}}} \Big)^{1-k} \|\eta\|_{a(\Omega\backslash K_{i})}^{2}
\leq 2D (1+\cfrac{1}{\Lambda}) \, \Big(1+\cfrac{\Lambda^{\frac{1}{2}}}{2D^{\frac{1}{2}}} \Big)^{1-k} \|\eta\|_{a}^{2}.
\end{equation}
By the definition of $\eta$ and the energy minimizing property of $\psi_j^{(i)}$, we have
\begin{equation*}
\|\eta\|_{a} = \| \psi_{j}^{(i)}-\tilde{\phi}_{j}^{(i)} \|_a \leq  2\| \tilde{\phi}_{j}^{(i)} \|_a
\leq 2D^{\frac{1}{2}} \| \phi_{j}^{(i)} \|_{s(K_i)}
\end{equation*}
where the last inequality follows from (\ref{eq:est}).

\noindent
{\bf Step 4}. We will prove the estimate (\ref{eq:recursive}). Let $\xi = 1-\chi_i^{k-1,k-2}$.
Then we see that $\xi \equiv 1$ in $\Omega\backslash K_{i,k-1}$
and $0 \leq \xi \leq 1$ otherwise. Then we have
\begin{align}
\label{bound}
\|\eta\|_{a(\Omega\backslash K_{i,k-1})}^{2} & \leq\int_{\Omega}\kappa\xi^{2}|\nabla\eta|^{2}
  =\int_{\Omega}\kappa\nabla\eta\cdot\nabla(\xi^{2}\eta)-2\int_{\Omega}\kappa\xi\eta\nabla\xi\cdot\nabla\eta.
\end{align}
We estimate the first term in (\ref{bound}).
For the function $\pi(\xi^2 \eta)$,
using Lemma \ref{lem:projection}, there exist $\gamma \in V$ such that $\pi(\gamma)=\pi(\xi^{2}\eta)$
and $\text{supp}(\gamma) \subset \text{supp}(\pi(\xi^2 \eta))$.
For any coarse element $K_m \subset \Omega\backslash K_{i,k-1}$, since $\xi \equiv 1$ on $K_m$, we have
\begin{equation*}
s_{m}(\xi^{2}\eta,\phi_{n}^{(m)})=0, \quad \forall n = 1,2, \cdots, l_m.
\end{equation*}
On the other hand, since $\xi \equiv 0$ in $K_{i,k-2}$, we have
\begin{equation*}
s_{m}(\xi^{2}\eta,\phi_{n}^{(m)})=0, \quad \forall n = 1,2, \cdots, l_m, \; \forall K_m \subset K_{i,k-2}.
\end{equation*}
From the above two conditions, we see that $\text{supp}(\pi(\xi^2 \eta)) \subset K_{i,k-1} \backslash K_{i,k-2}$,
and consequently
$\text{supp}(\gamma) \subset K_{i,k-1} \backslash K_{i,k-2}$.
Note that, since $\pi(\gamma)=\pi(\xi^{2}\eta)$, we have $\xi^2 \eta - \gamma \in \tilde{V}$.
We note also that $\text{supp}(\xi^2 \eta - \gamma) \subset \Omega \backslash K_{i,k-2}$.
By (\ref{eq:est}), the functions $\tilde{\phi}_j^{(i)}$ and $\xi^2 \eta - \gamma$ have disjoint supports,
so $a(\tilde{\phi}_j^{(i)}, \xi^2\eta - \gamma)=0$. Then, by the definition of $\eta$, we have
\begin{equation*}
a(\eta, \xi^2\eta - \gamma) = a(\psi_j^{(i)}, \xi^2\eta - \gamma).
\end{equation*}
By the construction of $\psi_j^{(i)}$, we have $a(\psi_j^{(i)}, \xi^2\eta - \gamma)=0$.
Then we can estimate the first term in (\ref{bound}) as follows
\begin{align*}
\int_{\Omega}\kappa\nabla\eta\cdot\nabla(\xi^{2}\eta) & =\int_{\Omega}\kappa\nabla\eta\cdot\nabla \gamma \\
 & \leq D^{\frac{1}{2}} \|\eta\|_{a(K_{i,k-1}\backslash K_{i,k-2})}
\|\pi(\xi^{2}\eta)\|_{s(K_{i,k-1}\backslash K_{i,k-2})}.
\end{align*}
For all coarse element $K \subset K_{i,k-1}\backslash K_{i,k-2}$, since $\pi(\eta)=0$,
we have
\begin{align*}
\|\pi(\xi^{2}\eta)\|_{s(K)}^{2} & =\|\pi(\xi^{2}\eta)\|_{s(K)}^{2}
  \leq \|\xi^{2}\eta\|_{s(K)}^{2}
  \leq (\cfrac{1}{\Lambda})\int_{K}\kappa|\nabla\eta|^{2}.
\end{align*}
Summing the above over all coarse elements $K \subset K_{i,k-1}\backslash K_{i,k-2}$, we have
\[
\|\pi(\xi^{2}\eta)\|_{s(K_{i,k-1}\backslash K_{i,k-2})}
\leq(\cfrac{1}{\Lambda})^{\frac{1}{2}}\|\eta\|_{a(K_{i,k-1}\backslash K_{i,k-2}))}.
\]
To estimate the second term in (\ref{bound}), using the spectral problem (\ref{spectralProblem_GMsFEM}),
\begin{align*}
2\int_{\Omega}\kappa\xi\eta\nabla\xi\cdot\nabla\eta & \leq2\|\eta\|_{s(\Omega\backslash K_{i,k-2})}\|\eta\|_{a(K_{i,k-1}\backslash K_{i,k-2})}\\
 & \leq\cfrac{2}{\Lambda^{\frac{1}{2}}}\|\eta\|_{a(K_{i,k-1}\backslash K_{i,k-2})}^2.
\end{align*}
Hence, by using the above results, (\ref{bound}) can be estimated as
\[
\|\eta\|_{a(\Omega\backslash K_{i,(k-1)})}^{2}\leq\cfrac{2D^{\frac{1}{2}}}{\Lambda^{\frac{1}{2}}}\|\eta\|_{a(K_{i,(k-1)}\backslash K_{i,(k-2)})}^{2}.
\]
By using the above inequality, we have
\begin{align*}
\|\eta\|_{a(\Omega\backslash K_{i,(k-2)})}^{2} &= \|\eta\|_{a(\Omega\backslash K_{i,(k-1)})}^{2} + \|\eta\|_{a(K_{i,(k-1)}\backslash K_{i,(k-2)})}^{2} \\
& \geq \Big(1+\cfrac{\Lambda^{\frac{1}{2}}}{2D^{\frac{1}{2}}} \Big) \|\eta\|_{a(\Omega\backslash K_{i,(k-1)})}^{2}
\end{align*}

This completes the proof.

\end{proof}

The above lemma shows the global basis is localizable.
Next, we use the above result to obtain an estimate of the error between the solution $u$ and the multiscale solution $u_{ms}$.

\begin{theorem}
\label{thm:conv1}
Let $u$ be the solution of (\ref{eq:finesol}) and $u_{ms}$ be
the multiscale solution of (\ref{eq:mssol}). Then we have
\[
\|u-u_{ms}\|_{a}\leq C\Lambda^{-\frac{1}{2}}
\| \tilde{\kappa}^{-\frac{1}{2}} f\|_{L^{2}}+C k^{d}E^{\frac{1}{2}} \|u_{glo}\|_s
\]
where $u_{glo}\in V_{glo}$ is the multiscale solution using global basis.
Moreover, if $k=O(log(\cfrac{\max\{\kappa\}}{H}))$ and $\chi_{i}$ are bilinear partition
of unity, we have
\[
\|u_{h}-u_{ms}\|_{a}\leq CH\Lambda^{-\frac{1}{2}} \| \kappa^{-\frac{1}{2}} f\|_{L^{2}(\Omega)}.
\]

\end{theorem}
\begin{proof}
We write $u_{glo}=\sum_{i=1}^N\sum_{j=1}^{l_i} c_{j}^{(i)}\psi_{j}^{(i)}$.
Then we define $v=\sum_{i=1}^N \sum_{j=1}^{l_i} c_{j}^{(i)}\psi_{j,ms}^{(i)}\in V_{ms}$. So, by the Galerkin orthogonality, we have
\begin{align*}
\|u-u_{ms}\|_{a} & \leq\|u-v\|_{a}
  \leq\|u-u_{glo}\|_{a}+\|\sum_{i=1}^N\sum_{j=1}^{l_{i}} c_{j}^{(i)}(\psi_{j}^{(i)}-\psi_{j,ms}^{(i)})\|_{a}.
\end{align*}
Recall that the basis functions $\psi_{j,ms}^{(i)}$ have supports in $K_{i,k}$. So, by Lemma \ref{lem:decay},
\begin{align*}
\|\sum_{i=1}^N \sum_{j=1}^{l_i} c_{j}^{(i)}(\psi_{j}^{(i)}-\psi_{j,ms}^{(i)})\|_{a}^{2} & \leq C k^{2d}\sum_{i=1}^N\|\sum_{j=1}^{l_{i}}c_{j}^{(i)}(\psi_{j}^{(i)}-\psi_{j,ms}^{(i)})\|_{a}^{2}\\
 & \leq C k^{2d}E \sum_{i=1}^N \| \sum_{j=1}^{l_i} c_{j}^{(i)} \phi_{j}^{(i)}\|_{s}^{2}\\
 & \leq C k^{2d}E \| u_{glo}\|_s^2
\end{align*}
where the equality follows from the orthogonality of the eigenfunctions in (\ref{spectralProblem_GMsFEM}).
We also note that, in the above estimate, we apply Lemma \ref{lem:decay} to the function $\sum_{j=1}^{l_i} c_{j}^{(i)}(\psi_{j}^{(i)}-\psi_{j,ms}^{(i)})$.
By using Lemma \ref{lem:global_estimate}, we obtain
\[
\|u_{h}-u_{ms}\|_{a}\leq C\Lambda^{-\frac{1}{2}}
\| \tilde{\kappa}^{-\frac{1}{2}} f\|_{L^{2}}+Ck^{d}E^{\frac{1}{2}} \|u_{glo}\|_s.
\]
This completes the proof for the first part of the theorem.



To proof the second inequality, we need to estimate the $s$-norm of the global solution $u_{glo}$. In particular,
\begin{align*}
\|u_{glo}\|_s^2 \leq \max\{\tilde{\kappa}  \} \|u_{glo}\|_{L^2(\Omega)}^2
 \leq C \kappa_0^{-1} \max\{\tilde{\kappa}  \} \|u_{glo}\|_{a}^2.
\end{align*}
Since $u_{glo}$ satisfies (\ref{eq:glosol}), we have
\begin{align*}
\|u_{glo}\|_a^2 = \int_\Omega f u_{glo}
 \leq \| \tilde{\kappa}^{-\frac{1}{2}} f\|_{L^{2}(\Omega)} \|u_{glo}\|_s.
\end{align*}
Therefore,  we have
\[
\|u_{glo}\|_s\leq C\kappa_0^{-1}\max\{\tilde{\kappa}  \} \|\tilde{\kappa}^{-\frac{1}{2}}f\|_{L^{2}(\Omega)}.
\]
Thus, to obtain the second inequality, we need to show that $\max\{\tilde{\kappa}  \} C k^d E^{\frac{1}{2}}$ is bounded.
Assuming the partition of unity functions are bilinear, then we have
\begin{equation*}
\max\{ \kappa\} H^{-2} Ck^d E^{\frac{1}{2}} = O(1).
\end{equation*}
Taking logarithm,
\begin{equation*}
d \log(k) + \frac{1-k}{2} \log( 1 + \frac{\Lambda^{\frac{1}{2}}}{2D^{\frac{1}{2}}}) + \log(\max{\kappa}) + \log(H^{-2}) = O(1).
\end{equation*}
Thus, we need $k=O(\log(\cfrac{\max\{\kappa\}}{H}))$. This completes the proof.
\end{proof}

We remark that the decay rate of the basis functions depends on the factor $1 + \frac{\Lambda^{\frac{1}{2}}}{2D^{\frac{1}{2}}}$.
Since $\Lambda$ is not small as eigenfunctions with small eigenvalues are used in the construction of auxiliary space,
we see that, when $D$ is not large, the decay is exponential.

\section{Numerical Result}
\label{sec:num}

In this section, we will present two numerical examples with two different high contrast media
to demonstrate the convergence of our proposed method.
We take the domain $\Omega = (0,1)^2$.
For the first numerical example, we consider the medium parameter $\kappa$ as shown in Figure \ref{fig:case1_kappa}
and assume that the fine mesh size
$h$ to be $1/400$. That is, the medium $\kappa$ has a $400\times 400$
resolution.
In this case, we consider the contrast of the medium is $10^4$ where the value of $\kappa$ is large in the red region.
The convergence history with various coarse mesh sizes $H$ are shown
in Table \ref{tab:case1_error}.
In all these simulations, we take the number of oversampling layer to be approximately
$4\log(1/H)/\log(1/10)$. Form Table \ref{tab:case1_error}, we can see the energy norm error converges in first order with respect to $H$
and the $L^2$ norm error converges in second order with respect to $H$. The first order convergence in the energy norm matches our theoretical bound.

\begin{figure}[ht!]
\centering
\includegraphics[scale=0.5]{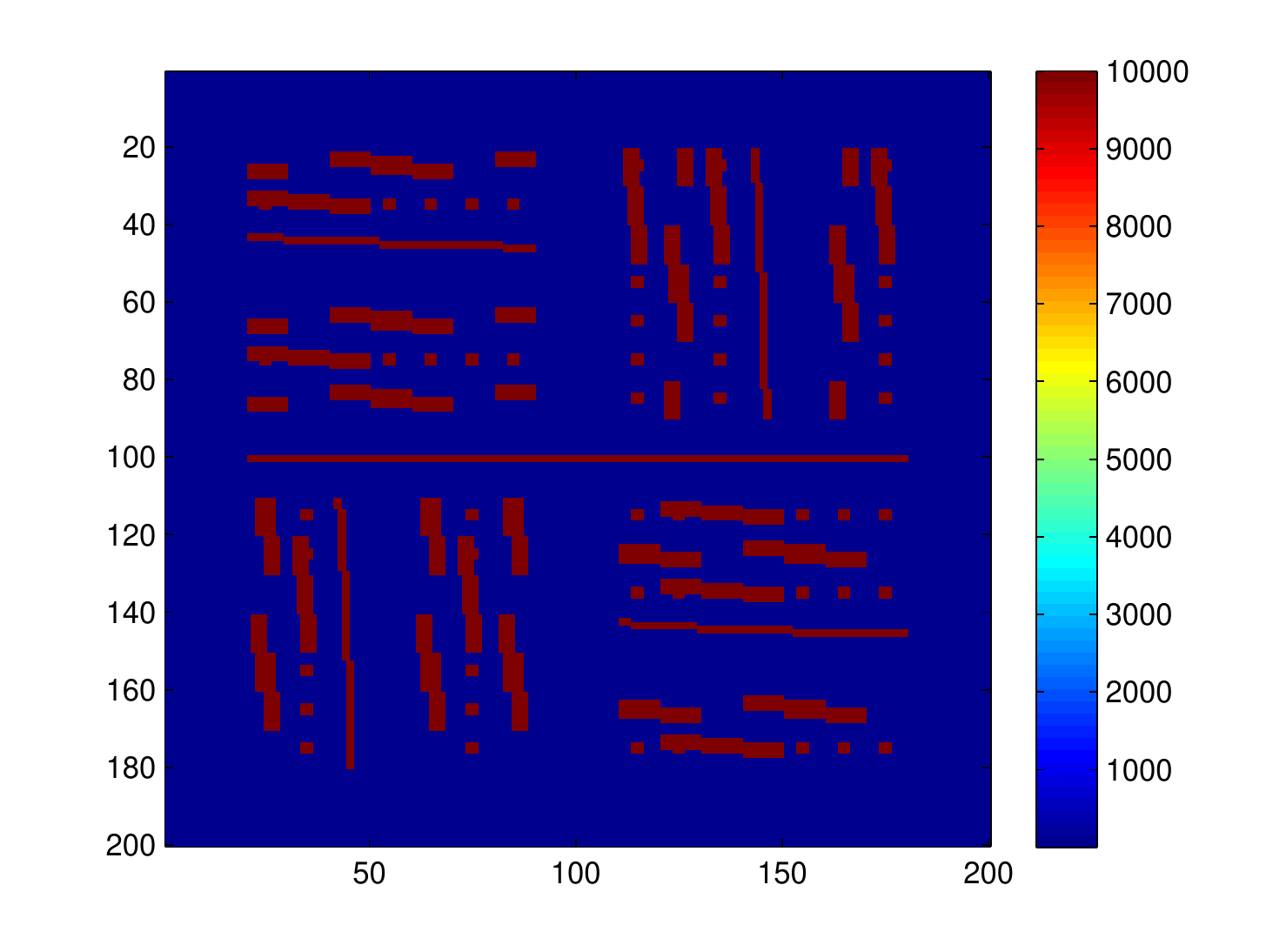}
\caption{The medium $\kappa$ for the test case $1$.}
\label{fig:case1_kappa}
\end{figure}

\begin{table}[ht!]
\centering
\begin{tabular}{|c|c|c|c|c|}
\hline
Number basis per element & H & \# oversampling coarse layers & $L_{2}$error & energy error\tabularnewline
\hline
3 & 1/10 & 4 & 2.62\% & 15.99\%\tabularnewline
\hline
3 & 1/20 & 6 (log(1/20)/log(1/10){*}4=5.20) & 0.51\% & 7.04\%\tabularnewline
\hline
3 & 1/40 & 7 (log(1/40)/log(1/10){*}4=6.41) & 0.11\% & 3.31\%\tabularnewline
\hline
3 & 1/80 & 8 (log(1/80)/log(1/10){*}4=7.61) & 0.0015\% & 0.17\%\tabularnewline
\hline
\end{tabular}

\caption{Numerical results with varying coarse grid size $H$ for the test case $1$.}
\label{tab:case1_error}
\end{table}

We emphasize that, in this example, we use $3$ basis functions per coarse region.
The reason of this is that the eigenvalue problem on each coarse region has $3$ small eigenvalues,
and according to our theory, we need to include the first three eigenfunctions in the auxiliary space.
As our theory shows, for a certain contrast value, one needs to use a large enough oversampling size
in order to obtain the desired convergence order. Moreover, for a fixed contrast value,
the results can be improved as the oversampling size increases.
On the other hand, for a fixed oversampling size, the performance of the scheme will deteriorate
as the medium contrast increases.
This is confirmed by our numerical evidence shown in Table \ref{tab:case1_compare}.

\begin{table}
\centering
\begin{tabular}{|c|c|c|c|c|}
\hline
 Layer $\backslash$ Contrast & 1e+3 & 1e+4 & 1e+5 & 1e+6\tabularnewline
\hline
3 & 48.74\% & 73.22\% & 87.83\% & 91.14\% \tabularnewline
\hline
4 & 19.12\% & 15.99\% & 26.68\% & 58.67\% \tabularnewline
\hline
5 & 3.70\% & 4.19\% & 7.17\% & 19.24\% \tabularnewline
\hline
\end{tabular}

\caption{Comparison of various number of oversampling layers and different contrast values for test case $1$.}
\label{tab:case1_compare}
\end{table}

\begin{figure}[ht!]
\centering
\includegraphics[scale=0.5]{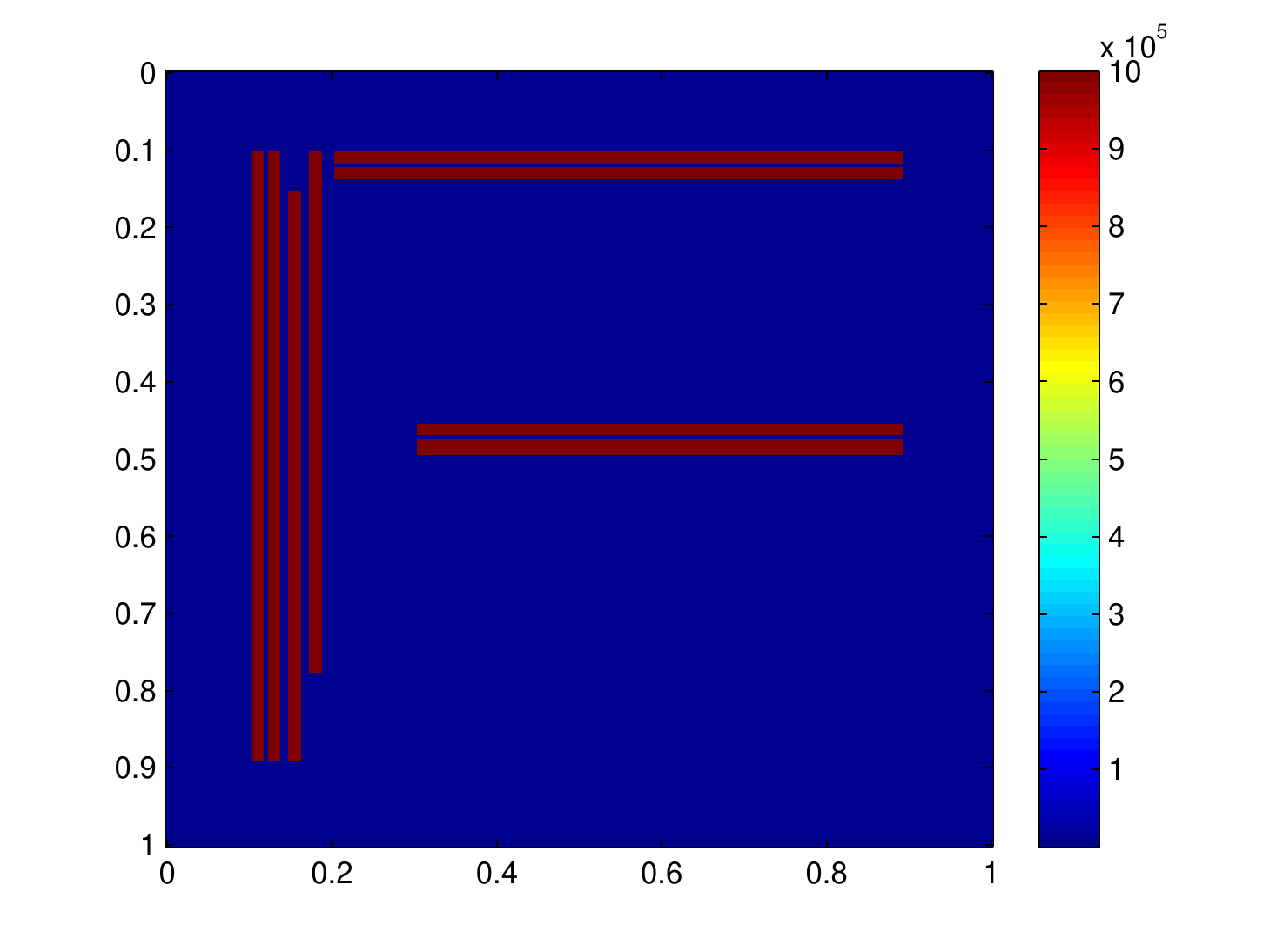}
\caption{The medium $\kappa$ for the test case $2$.}
\label{fig:case2_kappa}
\end{figure}

\begin{table}[ht!]
\centering

\begin{tabular}{|c|c|c|c|c|}
\hline
Number basis per element & H & \# oversampling coarse layers & $L_{2}$error & energy error\tabularnewline
\hline
4 & 1/10 & 6 & 1.55\% & 11.10\%\tabularnewline
\hline
4 & 1/20 & 8 (log(1/20)/log(1/10){*}6=7.8062) & 0.02\% & 0.59\%\tabularnewline
\hline
4 & 1/40 & 10 (log(1/40)/log(1/10){*}6=9.61) & 0.0042\% & 0.23\%\tabularnewline
\hline
\end{tabular}

\protect\caption{Numerical results with varying coarse grid size $H$ for the test case $2$.}
\label{tab:case2_error}
\end{table}

In our second test case, we consider
the medium parameter $\kappa$ defined in Figure~\ref{fig:case2_kappa}.
In this case, the medium has a contrast value of $10^6$, and
the fine grid size $h$ is $1/200$. We will show the convergence history in Table~\ref{tab:case2_error} using different choices of coarse mesh sizes.
For all simulations, the number of oversampling layer is approximately
$6\log(1/H)/\log(10)$, and we use $4$ multiscale basis functions per coarse block since there are $4$ small eigenvalues
for some coarse blocks.
Form Table~\ref{tab:case2_error}, we can see that the method achieves the theoretical convergence rate.
In Table~\ref{tab:case2_compare}, we compare the performance of the method with different choices of oversampling layers
and contrast values. We see that, for a fixed choice of oversampling layer, the error increases moderately with respect to the contrast value.
On the other hand, for a fixed contrast value, the error will improve as the number of oversampling layers increases.
We also see that, the error will be small once an enough number of oversampling layer is used.

\begin{table}
\centering
\begin{tabular}{|c|c|c|c|c|}
\hline
 Layer $\backslash$ Contrast & 1e+3 & 1e+4 & 1e+5 & 1e+6\tabularnewline
\hline
4 & 21.88\% & 47.58\% & 65.94\% & 81.47\% \tabularnewline
\hline
5 & 6.96\% & 22.49\% & 33.09\% & 47.19\% \tabularnewline
\hline
6 & 1.82\% & 4.33\% & 6.49\% & 11.09\% \tabularnewline
\hline
\end{tabular}

\caption{Comparison of various number of oversampling layers and different contrast values for test case $2$.}
\label{tab:case2_compare}
\end{table}

%


\section{Relaxed constraint energy minimizing generalized multiscale finite element method}
\label{sec:mod}

In this section, we consider a relaxed version of our method. In particular, we relax
the constraint in the minimization problem (\ref{eq:ms-min}). Instead of (\ref{eq:ms-min}),
we solve the following un-constrainted minimization problem: find $\psi_{j,ms}^{(i)} \in V_0(K_{i,m})$ such that
\begin{equation}
\label{eq:ms-min-new}
 \psi_{j,ms}^{(i)}=\text{argmin}\Big\{ a(\psi,\psi) + s(\pi \psi - \phi_j^{(i)},\pi \psi-\phi_j^{(i)}) \; | \; \psi \in V_{0}(K_{i,m}) \Big\}.
\end{equation}
This minimization problem is equivalent to the following variational formulation
\begin{equation}
\label{eq:ms-min-new-var}
a(\psi_{j,ms}^{(i)},v)+s(\pi(\psi_{j,ms}^{(i)}),\pi(v))=s(\phi_{j}^{(i)},\pi(v)), \quad \;\forall v\in V_{0}(K_{i,m}).
\end{equation}
The global multiscale basis function $\psi_j^{(i)} \in V$ is defined in a similar way, namely,
\begin{equation}
\label{eq:glo-min-new}
 \psi_{j}^{(i)}=\text{argmin}\Big\{ a(\psi,\psi) + s(\pi \psi - \phi_j^{(i)},\pi \psi-\phi_j^{(i)}) \; | \; \psi \in V \Big\},
\end{equation}
which is equivalent to the following variational form
\begin{equation}
\label{eq:glo-min-new-var}
a(\psi_{j}^{(i)},v)+s(\pi(\psi_{j}^{(i)}),\pi(v))=s(\phi_{j}^{(i)},\pi(v)), \quad \;\forall v\in V.
\end{equation}
We remark that the spaces $V_{glo}$, $V_{ms}$ and $\tilde{V}$ are defined as before.

Given a function $\tilde{v}\in\tilde{V}_{h}$, we have $\pi(\tilde{v})=0$ by the definition.
Therefore $a(\psi_{j}^{(i)},\tilde{v})=0,$ that is $\tilde{V}_{h}\subset V_{glo}^{\bot}$
and since $\text{dim}(V_{glo})=\text{dim}(V_{aux})$, we have $\tilde{V}_{h}=V_{glo}^{\bot}$.
Thus, we have $V_{h}=V_{glo}\oplus\tilde{V}_{h}$.
Using this property, the solution $u_{glo}$ of (\ref{eq:glosol})
satisfies Lemma \ref{lem:glosol}.

\subsection{Analysis}

In this section, we analyze the convergence of this method.
In the following, we prove a result similar to the one in Lemma \ref{lem:decay}, with the aim
of estimating the difference between $\psi_{j,ms}^{(i)}$ and $\psi_j^{(i)}$.

\begin{lemma}
\label{lem:new}
We consider the oversampled domain $K_{i,k}$ with $k\geq 2$.
That is, $K_{i,k}$ is an oversampled region by enlarging $K_i$ by $k$ coarse grid layers.
Let $\phi_j^{(i)} \in V_{aux}$ be a given auxiliary multiscale basis function.
We let $\psi_{j,ms}^{(i)}$ be the multiscale basis functions obtained in (\ref{eq:ms-min-new})
and let $\psi_{j}^{(i)}$ be the global multiscale basis functions obtained in (\ref{eq:glo-min-new}).
Then we have
\[
\|\psi_{j}^{(i)}-\psi_{j,ms}^{(i)}\|_{a}^{2}+\|\pi(\psi_{j}^{(i)}-\psi_{j,ms}^{(i)})\|_{s}^{2}\leq E
\Big(\|\psi_{j}^{(i)}\|_{a}^{2}+\|\pi(\psi_{j}^{(i)})\|_{s}^{2}\Big)
\]
where $E = 3(1+\Lambda^{-1}) \Big (1+(2(1+\Lambda^{-\frac{1}{2}}))^{-1}\Big)^{1-k}$.
\end{lemma}
\begin{proof}
By the definitions $\psi_{j,ms}^{(i)}$ and $\psi_j^{(i)}$ in (\ref{eq:ms-min-new-var}) and (\ref{eq:glo-min-new-var}), we have
\begin{align*}
a(\psi_{j,ms}^{(i)},v)+s(\pi(\psi_{j,ms}^{(i)}),\pi(v))&=s(\phi_{j}^{(i)},\pi(v)), \quad\;\;\forall v\in V_{0}(K_{i,k}), \\
a(\psi_{j}^{(i)},v)+s(\pi(\psi_{j}^{(i)}),\pi(v))&=s(\phi_{j}^{(i)},\pi(v)), \quad \;\;\forall v\in V.
\end{align*}
Subtracting the above two equations, we have
\[
a(\psi_{j}^{(i)}-\psi_{j,ms}^{(i)},v)+s(\pi(\psi_{j}^{(i)}-\psi_{j,ms}^{(i)}),\pi(v))=0
\]
for all $v\in V_{0}(K_{i,k})$. Taking $v = w-\psi_{j,ms}^{(i)}$ with $w\in V_0(K_{i,k})$ in the above relation, we have
\[
\|\psi_{j}^{(i)}-\psi_{j,ms}^{(i)}\|_{a}^{2}+\|\pi(\psi_{j}^{(i)}-\psi_{j,ms}^{(i)})\|_{s}^{2}\leq\|\psi_{j}^{(i)}-w\|_{a}^{2}+\|\pi(\psi_{j}^{(i)}-w)\|_{s}^{2},\quad\;\forall w\in V_{h,0}(K_{i,k}).
\]
Let $w=\chi^{k,k-1}_i \psi_{j}^{(i)}$ in the above relation, we have
\begin{equation}
\label{eq:new1}
\|\psi_{j}^{(i)}-\psi_{j,ms}^{(i)}\|_{a}^{2}+\|\pi(\psi_{j}^{(i)}-\psi_{j,ms}^{(i)})\|_{s}^{2}
\leq\|\psi_{j}^{(i)}-\chi_{i}^{k,k-1}\psi_{j}^{(i)}\|_{a}^{2}+\|\pi(\psi_{j}^{(i)}-\chi_{i}^{k,k-1}\psi_{j}^{(i)})\|_{s}^{2}.
\end{equation}
Next, we will estimate these two terms on the right hand side of (\ref{eq:new1}). We divide the proof into four steps.

\noindent
{\bf Step 1}: We will estimate the term $\|(1-\chi_{i}^{k,k-1})\psi_{j}^{(i)}\|^{2}_{a}$ in (\ref{eq:new1}). By the definition of the norm $\|\cdot\|_a$
and the fact that $\text{supp}( 1-\chi_{i}^{k,k-1} ) \subset \Omega\backslash K_{i,k-1}$,
we have
\begin{align*}
\|(1-\chi_{i}^{k,k-1})\psi_{j}^{(i)}\|^{2}_{a} & \leq 2\int_{\Omega\backslash K_{i,k-1}}\kappa|1-\chi_{i}^{k,k-1}|^{2}|\nabla\psi_{j}^{(i)}|^{2}+\int_{\Omega\backslash K_{i,k-1}}\kappa|\nabla\chi_{i}^{k,k-1}|^{2}|\psi_{j}^{(i)}|^{2}\\
 & \leq 2\Big(\|\psi_{j}^{(i)}\|_{a(\Omega\backslash K_{i,k-1})}^{2}+\|\psi_{j}^{(i)}\|_{s(\Omega\backslash K_{i,k-1})}^{2}\Big).
\end{align*}
We note that for each $K \in\mathcal{T}^H$, we have
\begin{align}
\|\psi_j^{(i)}\|_{s(K)}^2 &= \|(I-\pi)(\psi_j^{(i)})+\pi(\psi_j^{(i)})\|_{s(K)}^2 \nonumber\\
& = \|(I-\pi)(\psi_j^{(i)})\|_{s(K)}^2+\|\pi(\psi_j^{(i)})\|_{s(K)}^2 \nonumber\\
& \leq \Lambda^{-1} \|\psi_j^{(i)}\|_{a(K)}^2+\|\pi(\psi_j^{(i)})\|_{s(K)}^2.
\label{eq:s-norm_Bound}
\end{align}
Therefore, we have
\[
 \|(1-\chi_{i}^{k,k-1})\psi_{j}^{(i)}\|^2_a\leq 2\Big((1+\cfrac{1}{\Lambda})\|\psi_{j}^{(i)}\|_{a(\Omega\backslash K_{i,k-1})}^{2}+\|\pi(\psi_{j}^{(i)})\|_{s(\Omega\backslash K_{i,k-1})}^{2}\Big).
\]

\noindent
{\bf Step 2}: We will estimate the term $\|\pi \Big((1-\chi_{i}^{k,k-1})\psi_{j}^{(i)}\Big)\|^{2}_{s}$ in (\ref{eq:new1}).
Notice that
\begin{align*}
\|\pi \Big((1-\chi_{i}^{k,k-1})\psi_{j}^{(i)}\Big)\|^{2}_{s} &\leq \|\Big((1-\chi_{i}^{k,k-1})\psi_{j}^{(i)}\Big)\|^{2}_{s}
 \leq \|\psi_{j}^{(i)}\|^{2}_{s(\Omega\backslash K_{i,k-1})}
\end{align*}
By using \eqref{eq:s-norm_Bound}, we have
\[
 \|\pi \Big((1-\chi_{i}^{l,l-1})\psi_{j}^{(i)}\Big)\|^{2}_{s} \leq \Lambda^{-1} \|\psi_j^{(i)}\|_{a(\Omega\backslash K_{i,k-1})}^2
 +\|\pi(\psi^{(i)})\|_{s(\Omega\backslash K_{i,k-1})}^2.
\]

From the above two steps, we see that (\ref{eq:new1}) can be estimated as
\begin{equation}
\label{eq:new2}
\|\psi_{j}^{(i)}-\psi_{j,ms}^{(i)}\|_{a}^{2}+\|\pi(\psi_{j}^{(i)}-\psi_{j,ms}^{(i)})\|_{s}^{2}
\leq 3 (1+\Lambda^{-1}) \Big(  \|\psi_j^{(i)}\|_{a(\Omega\backslash K_{i,k-1})}^2+\|\pi(\psi_j^{(i)})\|_{s(\Omega\backslash K_{i,k-1})}^2\Big).
\end{equation}
Next we will estimate the right hand side of (\ref{eq:new2}).

\noindent
{\bf Step 3}: We will estimate $ \|\psi_j^{(i)}\|_{a(\Omega\backslash K_{i,k-1})}^2+\|\pi(\psi_j^{(i)})\|_{s(\Omega\backslash K_{i,k-1})}^2$.
We will show that this term can be estimated by $\|\psi_j^{(i)}\|^2_{(a({K_{i,k-1}\backslash K_{i,k-2}})} + \|\pi(\psi_j^{(i)})\|^2_{s({K_{i,k-1}\backslash K_{i,k-2}})}$.
By using the variational form (\ref{eq:glo-min-new-var}) and using the test function $(1-\chi_{i}^{k-1,k-2})\psi_{j}^{(i)}$, we have
\begin{equation}
\label{eq:new3}
a(\psi_{j}^{(i)},(1-\chi_{i}^{k-1,k-2})\psi_{j}^{(i)})+s(\pi(\psi_{j}^{(i)}),\pi\left((1-\chi_{i}^{k-1,k-2})\psi_{j}^{(i)}\right))=s(\phi_{j}^{(i)},\pi\left((1-\chi_{i}^{k-1,k-2})\psi_{j}^{(i)}\right))=0
\end{equation}
where the last equality follows from the facts that $\text{supp}(1-\chi_{i}^{k-1,k-2}) \subset \Omega \backslash K_{i,k-2}$
and $\text{supp}(\phi_j^{(i)}) \subset K_i$. Note that
\begin{equation*}
a(\psi_{j}^{(i)},(1-\chi_{i}^{k-1,k-2})\psi_{j}^{(i)})
= \int_{\Omega \backslash K_{i,k-2}} \kappa \nabla \psi_j^{(i)} \cdot \nabla ((1-\chi_{i}^{k-1,k-2})\psi_{j}^{(i)}),
\end{equation*}
so we have
\begin{equation*}
a(\psi_{j}^{(i)},(1-\chi_{i}^{k-1,k-2})\psi_{j}^{(i)})
= \int_{\Omega \backslash K_{i,k-2}} \kappa (1-\chi_{i}^{k-1,k-2}) |\nabla \psi_j^{(i)}|^2
- \int_{\Omega \backslash K_{i,k-2}} \kappa \psi_j^{(i)} \nabla \chi_{i}^{k-1,k-2} \cdot \nabla \psi_j^{(i)}.
\end{equation*}
Consequently, we have
\begin{equation}
\label{eq:bound1}
\begin{split}
\|\psi_j^{(i)}\|_{a(\Omega\backslash K_{i,k-1})}^2
&\leq \int_{\Omega \backslash K_{i,k-2}} \kappa (1-\chi_{i}^{k-1,k-2}) |\nabla \psi_j^{(i)}|^2 \\
&= a(\psi_{j}^{(i)},(1-\chi_{i}^{k-1,k-2})\psi_{j}^{(i)}) + \int_{\Omega \backslash K_{i,k-2}} \kappa \psi_j^{(i)} \nabla \chi_{i}^{k-1,k-2} \cdot \nabla \psi_j^{(i)} \\
&\leq a(\psi_{j}^{(i)},(1-\chi_{i}^{k-1,k-2})\psi_{j}^{(i)}) +  \|\psi_j^{(i)}\|_{a(K_{i,k-1}\backslash K_{i,k-2})}\|\psi_j^{(i)}\|_{s(K_{i,k-1} \backslash K_{i,k-2})}.
\end{split}
\end{equation}

Next, we note that, since $\chi_j^{k-1,k-2} \equiv 0$ in $\Omega \backslash K_{i,k-1}$, we have
\begin{equation*}
s(\pi(\psi_{j}^{(i)}),\pi\left((1-\chi_{i}^{k-1,k-2})\psi_{j}^{(i)}\right))
= \|\pi(\psi_{j}^{(i)})\|_{s(\Omega\backslash K_{i,k-1})}^{2} +
\int_{K_{i,k-1}\backslash K_{i,k-2}} \tilde{\kappa} \pi(\psi_{j}^{(i)})\pi((1-\chi_{i}^{k-1,k-2})\psi_{j}^{(i)})
\end{equation*}
Thus, we have
\begin{align}
&\|\pi(\psi_{j}^{(i)})\|_{s(\Omega\backslash K_{i,k-1})}^{2} \nonumber \\
=&s(\pi(\psi_{j}^{(i)}),\pi\left((1-\chi_{i}^{k-1,k-2})\psi_{j}^{(i)}\right))
-\int_{K_{i,k-1}\backslash K_{i,k-2}} \tilde{\kappa} \pi(\psi_{j}^{(i)})\pi((1-\chi_{i}^{k-1,k-2})\psi_{j}^{(i)}) \nonumber \\
 \leq& s(\pi(\psi_{j}^{(i)}),\pi\left((1-\chi_{i}^{k-1,k-2})\psi_{j}^{(i)}\right)) + \|\psi_j^{(i)}\|_{s({K_{i,k-1}\backslash K_{i,k-2}})}\|\pi (\psi_j^{(i)})\|_{s({K_{i,k-1}\backslash K_{i,k-2}})}.
 \label{eq:bound2}
\end{align}

Finally, summing (\ref{eq:bound1}) and (\ref{eq:bound2}) and using (\ref{eq:new3}), we have
\begin{align}
&\|\psi_{j}^{(i)}\|_{a(\Omega\backslash K_{i,k-1})}^{2}+\|\pi(\psi_{j}^{(i)})\|_{s(\Omega\backslash K_{i,k-1})}^{2} \nonumber \\
 \leq & \|\psi_j^{(i)}\|_{s({K_{i,k-1}\backslash K_{i,k-2}})}\Big(\|\pi(\psi_j^{(i)})\|_{(s({K_{i,k-1}\backslash K_{i,k-2}})} + \|\psi_j^{(i)}\|_{(a({K_{i,k-1}\backslash K_{i,k-2}})}\Big)\nonumber \\
\leq & 2(1+\Lambda^{-\frac{1}{2}})(\|\pi(\psi_j^{(i)})\|^2_{s({K_{i,k-1}\backslash K_{i,k-2}})} + \|\psi_j^{(i)}\|^2_{(a({K_{i,k-1}\backslash K_{i,k-2}})}) \label{eq:bound3}
\end{align}
where the last inequality follows from (\ref{eq:s-norm_Bound}).

\noindent
{\bf Step 4}: We will show that $ \|\psi_j^{(i)}\|_{a(\Omega\backslash K_{i,k-1})}^2+\|\pi(\psi_j^{(i)})\|_{s(\Omega\backslash K_{i,k-1})}^2$
can be estimated by $ \|\psi_j^{(i)}\|_{a(\Omega\backslash K_{i,k-2})}^2+\|\pi(\psi_j^{(i)})\|_{s(\Omega\backslash K_{i,k-2})}^2$. This recursive property
is crucial in our convergence estimate. To do so, we note that
 By using this, we have
\begin{eqnarray*}
&&\|\psi_{j}^{(i)}\|_{a(\Omega\backslash K_{i,k-2})}^{2}+\|\pi(\psi_{j}^{(i)})\|_{s(\Omega\backslash K_{i,k-2})}^{2} \\
& = & \|\psi_{j}^{(i)}\|_{a(\Omega\backslash K_{i,k-1})}^{2}+\|\pi(\psi_{j}^{(i)})\|_{s(\Omega\backslash K_{i,k-1})}^{2}
  +\|\psi_{j}^{(i)}\|_{a(K_{i,k-1}\backslash K_{i,k-2})}^{2}+\|\pi(\psi_{j}^{(i)})\|_{s(K_{i,k-1}\backslash K_{i,k-2})}^{2}\\
 & \geq &\Big (1+(2(1+\Lambda^{-\frac{1}{2}}))^{-1}\Big) (\|\psi_{j}^{(i)}\|_{a(\Omega\backslash K_{i,k})}^{2}+\|\pi(\psi_{j}^{(i)})\|_{s(\Omega\backslash K_{i,k})}^{2})
\end{eqnarray*}
where we used (\ref{eq:bound3}) in the last inequality. Using the above inequality recursively, we have
\[
\|\psi_{j}^{(i)}\|_{a(\Omega\backslash K_{i,k-1})}^{2}+\|\pi(\psi_{j}^{(i)})\|_{s(\Omega\backslash K_{i,k-1})}^{2}
\leq\Big (1+(2(1+\Lambda^{-\frac{1}{2}}))^{-1}\Big)^{1-k}(\|\psi_{j}^{(i)}\|_{a}^{2}+\|\pi(\psi_{j}^{(i)})\|_{s}^{2}).
\]
\end{proof}

Finally, we state and prove the convergence.

\begin{theorem}
\label{thm:conv2}
Let $u$ be the solution of (\ref{eq:finesol}) and $u_{ms}$ be
the multiscale solution of (\ref{eq:mssol}). Then we have
\[
\|u-u_{ms}\|_{a}\leq C\Lambda^{-\frac{1}{2}}
\| \tilde{\kappa}^{-\frac{1}{2}} f\|_{L^{2}}+C k^{d}E^{\frac{1}{2}} (1+D)^{\frac{1}{2}}\|u_{glo}\|_s
\]
where $u_{glo}\in V_{glo}$ is the multiscale solution using global basis.
Moreover, if $k=O(log(\cfrac{\max\{\kappa\}}{H}))$ and $\chi_{i}$ are bilinear partition
of unity, we have
\[
\|u-u_{ms}\|_{a}\leq CH\Lambda^{-\frac{1}{2}} \| \kappa^{-\frac{1}{2}} f\|_{L^{2}(\Omega)}.
\]

\end{theorem}

\begin{proof}
The proof follows the same procedure as the proof of Theorem \ref{thm:conv1}.
We write $u_{glo}=\sum_{i=1}^N \sum_{j=1}^{l_i} c_{ij}\psi_{j}^{(i)}$
and define $v=\sum_{i=1}^N \sum_{j=1}^{l_i} c_{ij}\psi_{j,ms}^{(i)}$.
It suffices to estimate $\|u_{glo}-v\|_a$. By Lemma \ref{lem:new},
\begin{align*}
\|u_{glo}-v\|_a^2 & \leq C k^{2d}\sum_{i=1}^N\|\sum_{j=1}^{l_{i}}c_{ij}(\psi_{j}^{(i)}-\psi_{j,ms}^{(i)})\|_{a}^{2}\\
 & \leq C k^{2d}E \sum_{i=1}^N \sum_{j=1}^{l_i} \|c_{ij} \phi_{j}^{(i)}\|_{s}^{2} \\
  & = C k^{2d}E \sum_{i=1}^N \sum_{j=1}^{l_i} (c_{ij})^{2}
\end{align*}
since $\|\phi_{j}^{(i)}\|_{s}=1$. Notice that, in the above, we use Lemma \ref{lem:new}
to the function $\sum_{j=1}^{l_{i}}c_{ij}(\psi_{j}^{(i)}-\psi_{j,ms}^{(i)})$.

Note that, we have
\begin{equation*}
\pi u_{glo}=\sum_{i=1}^N \sum_{j=1}^{l_i} c_{ij}\pi\psi_{j}^{(i)}.
\end{equation*}
So, we obtain
\begin{equation*}
s(\pi u_{glo}, \phi_k^{(l)}) = \sum_{i=1}^N \sum_{j=1}^{l_i} c_{ij} s(\pi\psi_{j}^{(i)},\phi_k^{(l)}).
\end{equation*}
Using the variational problem (\ref{eq:glo-min-new-var}), we have
\begin{equation*}
s(\pi u_{glo}, \phi_k^{(l)}) = \sum_{i=1}^N \sum_{j=1}^{l_i} c_{ij} \Big( s(\pi(\psi_{j}^{(i)}),\pi(\psi_{k}^{(l)})) + a(\psi_{j}^{(i)},\psi_{k}^{(l)}) \Big).
\end{equation*}
Let $b_{lk} = s(\pi u_{glo}, \phi_k^{(l)})$ and $\vec{b} = (b_{lk})$.
We have
\begin{equation}
\label{eq:A}
\| \vec{c}\|_{2} \leq \| A^{-1} \|_2 \, \| \vec{b}\|_2
\end{equation}
where $A\in\mathbb{R}^{p\times p}$ is the matrix representation
of the bilinear form $\Big( s(\pi(\psi_{j}^{(i)}),\pi(\psi_{k}^{(l)})) + a(\psi_{j}^{(i)},\psi_{k}^{(l)}) \Big)$
where $p = \sum_{i=1}^N l_i$, and $\vec{c}=(c_{ij})$.
We will derive a bound for the largest eigenvalue of $A^{-1}$.
For a given $\vec{c}\in\mathbb{R}^p$, we define an auxiliary function $\phi = \sum_{i=1}^{N} \sum_{j=1}^{l_i} c_{ij} \phi_j^{(i)} \in V_{aux}$.
Using the variational problem (\ref{eq:glo-min-new-var}), there is $\psi\in V$ such that
\begin{equation}
\label{eq:thm2}
a(\psi, w) + s(\pi\psi, \pi w) = s(\phi, \pi w), \quad \forall w\in V
\end{equation}
and $\psi = \sum_{i=1}^{N} \sum_{j=1}^{l_i} c_{ij} \psi_j^{(i)}$.
Using the given $\phi \in V_{aux}$, by Lemma \ref{lem:infsup}, there is $z\in V$ such that
\begin{equation*}
\pi z = \phi, \quad \| z\|^2_a \leq D \|\phi\|^2_s.
\end{equation*}
Taking $w=z$ in (\ref{eq:thm2}),
\begin{equation*}
a(\psi, z) + s(\pi\psi, \pi z) = s(\phi, \pi z).
\end{equation*}
Notice that $s(\phi, \pi z) = s(\phi,\phi) = \|\vec{c}\|_2^2$. Thus,
\begin{equation*}
\| \vec{c}\|_2^2 = a(\psi, z) + s(\pi\psi, \phi)
\leq \| \psi\|_a \|z\|_a + \| \pi\psi\|_s \, \|\phi\|_s
\leq (1 + D)^{\frac{1}{2}} \| \phi\|_s \, \Big( \|\psi\|_a^2 + \| \pi\psi\|_s^2 \Big)^{\frac{1}{2}}
\end{equation*}
From the above, we see that the largest eigenvalue of $A^{-1}$ is $(1+D)^{\frac{1}{2}}$. So, we can estimate (\ref{eq:A}) as
\begin{equation}
\label{eq:A1}
\| \vec{c}\|_{2}^2 \leq  (1+D) \, \| \vec{b}\|^2_2 \leq  (1+D) \| u_{glo}\|_s^2.
\end{equation}
Using (\ref{eq:A1}), we conclude that
\begin{equation*}
\|u_{glo}-v\|_a^2 \leq Ck^{2d} E (1+D) \| u_{glo}\|_s^2.
\end{equation*}
The rest of the proof follows from the proof of Theorem \ref{thm:conv1}.

\end{proof}

Finally, we remark that the above theorem provides an improved bound
compared with Theorem \ref{thm:conv1}, since the convergence rate
in Theorem \ref{thm:conv2} is independent of the constant $D$.
This is confirmed by our numerical results presented next.


\subsection{Numerical Result}

In this section, we present numerical results to show the performance
of the relaxed version of the method.
As predicted by the theory, the relaxed version
of the method is more robust with respect to the contrast.
We will consider two test cases, which are the same as those considered in Section \ref{sec:num}.
First, in Table~\ref{tab:case1_error_new},
we show the errors for the first test case
using different choices of coarse mesh sizes.
We clearly see that the method gives the predicted convergence rate
since we have included enough eigenfunctions in the auxiliary space.
More importantly, by comparing to the similar test case
 in Table~\ref{tab:case1_error},
we see that the relaxed version is able to improve the number of oversampling layers.
In particular, we see that one needs fewer oversampling layers
and obtains much better results.

\begin{table}[ht!]
\centering

\begin{tabular}{|c|c|c|c|c|}
\hline
Number basis per $K$ & H & \# oversampling coarse layers & $L_{2}$error & energy error\tabularnewline
\hline
3 & 1/10 & 3 & 0.33\% & 3.73\%\tabularnewline
\hline
3 & 1/20 & 4 (log(1/20)/log(1/10){*}3=3.9031) & 0.047\% & 1.17\%\tabularnewline
\hline
3 & 1/40 & 5 (log(1/40)/log(1/10){*}3=4.8062) & 0.010\% & 0.47\%\tabularnewline
\hline
3 & 1/80 & 6 (log(1/40)/log(1/10){*}3=5.7093) & 0.0015\% & 0.15\%\tabularnewline
\hline
\end{tabular}

\protect\caption{Numerical result for the test case $1$ with the relaxed method.}
\label{tab:case1_error_new}
\end{table}

In Table~\ref{tab:case1_compare_new}, we show the performance of the relaxed version
with respect to the relation between contrast values and number of oversampling layers.
From the results, we see that the relaxed version needs a much smaller number
of oversampling layers in order to achieve a good result.
In particular, for a given oversampling layer, it can handle a much larger contrast value.
This confirms the theoretical estimates.
We performed a similar computation for the test case $2$
and obtain the same conclusion. For the numerical results, see
Table~\ref{tab:case2_error_new} and Table~\ref{tab:case2_compare_new}

\begin{table}
\centering
\begin{tabular}{|c|c|c|c|c|}
\hline
 Layer $\backslash$ Contrast & 1e+4 & 1e+6 & 1e+8 & 1e+10 \tabularnewline
\hline
3 & 3.73\% & 3.89\% & 11.99\% & 65.19\% \tabularnewline
\hline
4 & 3.72\% & 3.72\% & 3.73\% & 5.14\% \tabularnewline
\hline
5 & 3.72\% & 3.72\% & 3.73\% & 3.72\% \tabularnewline
\hline
\end{tabular}

\caption{Comparison for the test case $1$ with the relaxed method.}
\label{tab:case1_compare_new}
\end{table}

\begin{table}[ht!]
\centering

\begin{tabular}{|c|c|c|c|c|}
\hline
Number basis per element & H & \# oversampling coarse layers & $L_{2}$error & energy error\tabularnewline
\hline
4 & 1/10 & 4 & 0.11\% & 1.50\%\tabularnewline
\hline
4 & 1/20 & 6 (log(1/20)/log(1/10){*}4=5.2041) & 0.021\% & 0.57\%\tabularnewline
\hline
4 & 1/40 & 7 (log(1/40)/log(1/10){*}4=6.4082) & 0.0042\% & 0.23\%\tabularnewline
\hline
\end{tabular}

\protect\caption{Numerical result for the test case $2$ with the relaxed method.}
\label{tab:case2_error_new}
\end{table}

\begin{table}
\centering
\begin{tabular}{|c|c|c|c|}
\hline
 Layer $\backslash$ Contrast & 1e+6 & 1e+8 & 1e+10\tabularnewline
\hline
3 & 5.00\% & 41.11\% & 83.26\% \tabularnewline
\hline
4 & 1.50\% & 1.50\% & 11.88\% \tabularnewline
\hline
5 & 1.50\% & 1.50\% & 1.50\% \tabularnewline
\hline
\end{tabular}

\caption{Comparison for the test case $2$ with the relaxed method.}
\label{tab:case2_compare_new}
\end{table}

Finally, we test the performance with different choices of eigenfunctions
in the auxiliary space, and the results are shown in Table~\ref{tab:case2_basis_new}.
We consider the second medium parameter $\kappa$.
For this medium, there are $4$ high-contrast channels in some coarse blocks,
and therefore we need to use $4$ eigenfunctions in the auxiliary space.
As predicted by our theory, using less eigenfunctions will result
in a poor decay in the multiscale basis functions, the hence poor performance of the scheme.
This fact is confirmed by using one, two or three basis functions per coarse blocks.
We see that using $4$ basis functions will significantly improve the performance.
We also note that, using more than $4$ basis functions
will not necessarily improve the result further.

\begin{table}[ht!]
\centering

\begin{tabular}{|c|c|c|c|c|}
\hline
Number basis per element & H & \# oversampling coarse layers & $L_{2}$error & energy error\tabularnewline
\hline
1 & 1/10 & 4 & 77.30\% & 87.07\%\tabularnewline
\hline
2 & 1/10 & 4 & 30.21\% & 49.66\%\tabularnewline
\hline
3 & 1/10 & 4 & 24.27\% & 44.46\%\tabularnewline
\hline
4 & 1/10 & 4 & 0.11\% & 1.50\%\tabularnewline
\hline
5 & 1/10 & 4 & 0.08\% & 1.26\%\tabularnewline
\hline
\end{tabular}

\protect\caption{Using various numbers of basis functions for the test case $2$.}
\label{tab:case2_basis_new}
\end{table}

\section{Conclusions}
\label{sec:conclusions}

In this paper, we propose Constraint Energy Minimizing GMsFEM for solving
flow equations in high-contrast media.
The proposed method first constructs
an auxiliary space, which uses eigenvectors corresponding to small
eigenvalues in the local spectral problem.
This space contains
the subgrid information, which can not be localized and
it is a minimal dimensional space that one needs for preconditioning
and obtaining the errors that do not depend on the contrast.
Next, using local constraint energy minimizing construction in the
oversampled domain, we construct multiscale basis functions.
The constraint consists of some type of orthogonality with respect
to the auxiliary space. The choice of the auxiliary space is
 important to guarantee that the method converges as we decrease the mesh
size and the convergence is independent of the contrast if the
oversampling domain size is appropriately selected. Our main
theorem shows that the convergence depends on the oversampling domain
size that depends on the contrast.
We note that this is achieved with a minimal dimensional coarse space.
To remove this effect, we propose
a relaxation in imposing the constraint.
In our numerical results, we vary the number of oversampling layers, the contrast, the coarse-mesh size, and the number of auxiliary multiscale basis functions.
Our numerical results show that one needs a minimum number of auxiliary basis
functions to provide a good accuracy, which does not depend on the contrast.
The numerical results are presented,
which confirm our theoretical findings.

\section{Acknowledgements}

EC and YE would like to acknowledge the support of Hausdorff Institute
of Mathematics and Institute for Pure and Applied Mathematics for hosting
their long-term visits. WTL would like to acknowledge the support of
 Institute for Pure and Applied Mathematics for his long-term visits.

\bibliographystyle{plain}
\bibliography{references,references1,references_outline}

\end{document}